\documentclass[11pt]{amsart}

\usepackage{amsmath}
\usepackage{amssymb}
\usepackage{mathrsfs}
\usepackage{mathabx}
\usepackage{xcolor, enumitem, comment, todonotes}
\usepackage[all]{xy}
 \usepackage{url}
 \usepackage{amsthm}
 \usepackage{thmtools}
\usepackage{thm-restate}
\usepackage{hyperref,soul}

\usepackage{quiver}
\usepackage{centernot}

\theoremstyle{plain}
\newtheorem*{theorem*}{Theorem}

\newtheorem{theorem}{Theorem}[section]
\newtheorem{claim}[theorem]{Claim}

\newtheorem{lemma}[theorem]{Lemma}

\newtheorem{proposition}[theorem]{Proposition}
\newtheorem{corollary}[theorem]{Corollary}

\theoremstyle{definition}
\newtheorem{definition}[theorem]{Definition}

\newtheorem{question}[theorem]{Question}

\newtheorem*{cor*}{Corollary}
\theoremstyle{remark}
\newtheorem{remark}[theorem]{Remark}

\newtheorem{fact}[theorem]{Fact}

\newcommand\Fcal{\mathcal{F}}

\newcommand\Wcal{\mathcal{W}}
\newcommand\Vcal{\mathcal{V}}

\newcommand\RK{\mathrm{RK}}

\newcommand\Abb{\mathbb{A}}
\newcommand\Pbb{\mathbb{P}}
\newcommand\Qbb{\mathbb{Q}}
\newcommand\Ht{\mathrm{ht}}
\newcommand\Ucal{\mathcal{U}}

\newcommand\Th{{}^{th}}

\newcommand\MA{\mathrm{MA}}
\newcommand\PFA{\mathrm{PFA}}

\newcommand\Seq[1]{\langle #1 \rangle}

\newcount\skewfactor
\def\mathunderaccent#1#2 {\let\theaccent#1\skewfactor#2
\mathpalette\putaccentunder}
\def\putaccentunder#1#2{\oalign{$#1#2$\crcr\hidewidth
\vbox to.2ex{\hbox{$#1\skew\skewfactor\theaccent{}$}\vss}\hidewidth}}

\def\smallbox#1{\leavevmode\thinspace\hbox{\vrule\vtop{\vbox
   {\hrule\kern1pt\hbox{\vphantom{\tt/}\thinspace{\tt#1}\thinspace}}
   \kern1pt\hrule}\vrule}\thinspace}

\DeclareMathOperator{\dom}{dom}
\DeclareMathOperator{\supp}{Supp}

\DeclareMathOperator{\Add}{Add} 
\DeclareMathOperator{\Pri}{Prikry}

\newcommand{\cf}{{\rm cf}}
\newcommand{\fin}{\text{fin}}

\author{Tom Benhamou}
\address[Benhamou]{Department of Mathematics, Rutgers University, New Brunswick, NJ USA}
\email{tom.benhamou@rutgers.edu}
\author[J. Tatch Moore]{Justin Tatch Moore}
\address[Moore]{Department of Mathematics, Cornell University, Ithaca, NY, USA}
\email{justin@math.cornell.edu}
\author{Luke Serafin}
\address[Serafin]{Department of Mathematics, Cornell University, Ithaca, NY, USA}
\email{lss255@cornell.edu}

\thanks{The research of the first author was supported by the National Science Foundation under Grant
No. DMS-2346680;
the research of the second and third authors was supported by NSF grants DMS-2153975 and DMS-2451350.}

\title{Ultrafilters over Successor Cardinals and the Tukey Order}

\begin{document}

\begin{abstract}
We study ultrafilters on regular uncountable cardinals, with a primary focus on $\omega_1$, and particularly in relation to the Tukey order on directed sets.
Results include the independence from ZFC of the assertion that every uniform ultrafilter over $\omega_1$ is Tukey-equivalent to $[2^{\aleph_1}]^{<\omega}$, and
for each cardinal $\kappa$ of uncountable cofinality, a new
construction of a
uniform ultrafilter over $\kappa$ which extends the club filter and is Tukey-equivalent
to $[2^\kappa]^{<\omega}$.
We also analyze Todorcevic's ultrafilter $\Ucal(T)$ under PFA, proving that it is Tukey-equivalent to
$[2^{\aleph_1}]^{<\omega}$ and that it is minimal in the Rudin-Keisler order with respect to being a uniform ultrafilter over $\omega_1$.
We prove that, unlike PFA, $\MA_{\omega_1}$ is consistent with the existence
of a coherent Aronszajn tree $T$ for which $\Ucal(T)$ extends the club filter.
A number of other results are obtained
concerning the Tukey order on uniform ultrafilters and on uncountable
directed systems.
\end{abstract}

\maketitle
\section{Introduction}
Ultrafilters, and particularly their cofinalities and combinatorial properties, are of special interest in several areas of mathematics such as topology, combinatorics, group theory; and more centrally to model theory, mathematical logic, and set theory.
In this paper we deal with several fundamental questions concerning uniform ultrafilters over regular uncountable cardinals in general and $\omega_1$ in particular.
Recall that an ultrafilter over $\kappa$ is \emph{uniform} if all of its elements have cardinality $\kappa$---hence it is not isomorphic to a trivial extension of an ultrafilter over a smaller cardinal.

Our results were motivated by the following longstanding open problem of Kunen:
\begin{question}[Kunen]
Is it consistent that there is a uniform ultrafilter over $\omega_1$ which is generated by fewer than $2^{\aleph_1}$-many sets?
\end{question}
\noindent
It is also natural to pose this question for any uncountable cardinal $\kappa$;
we will refer to this variant as \emph{Kunen's problem at $\kappa$}.

There are several known methods to obtain ultrafilters over $\omega$ which are generated by fewer than $2^{\aleph_0}$-many elements.
Perhaps the most basic of them is Kunen's method~\cite{Kunen} to iterate Mathias forcing with
respect to an ultrafilter.
In unpublished work, Carlson generalized Kunen's 
construction to produce ultrafilters with small
generating sets over supercompact cardinals.
However, this method
cannot be straightforwardly adapted to produce such
ultrafilters over small uncountable cardinals.
Recently, Raghavan and Shelah have shown that
Kunen's problem at $\aleph_{\omega+1}$ and at $2^{\aleph_0}$ have positive answers modulo a large
cardinal hypothesis \cite{RagShel} (in the latter
model $2^{\aleph_0}$ is weakly inaccessible).
Still, new methods seem to be required to yield solutions to Kunen's problem at successors of regular cardinals.

Kunen's problem can be viewed as asking whether uniform
ultrafilters over $\omega_1$ are necessarily maximally
complicated, at least when measured by their
\emph{character}---the number of elements
which are required to generate them.
There are other natural notions of complexity on
ultrafilters which are finer.
The \emph{Tukey order} is defined and studied in the wider generality of directed sets,
and originated in the study of Moore-Smith convergence of nets from topology
(definitions and notation are reviewed in section~\ref{Section: Preliminaries} below).
The basic theory was set up by Tukey~\cite{Tukey40} in the 1940s, then further studied by Schmidt and Isbell~\cite{Schmidt55,Isbell65}.

Tukey showed that if
$\kappa$ is an infinite cardinal, the collection $[\kappa]^{<\omega}$ of all finite subsets of $\kappa$ ordered by inclusion
serves as an important benchmark in the Tukey order:
if $D$ is any directed set of cardinality at most $\kappa$, $D \leq_T [\kappa]^{<\omega}$.
A directed set $D$ of cardinality $\kappa$ such that $D \equiv_T [\kappa]^{<\omega}$, is said to be \emph{Tukey-top}.
Isbell~\cite{Isbell65} and, independently, Juh\'asz~\cite{Juhasz67}
constructed Tukey-top ultrafilters over any cardinal $\kappa$
using independent families.
Isbell posed what came to be known as \emph{Isbell's problem}: is every ultrafilter over $\omega$ Tukey-top?

While several constructions of Tukey-top ultrafilters are known~\cite{tomFanxin,Milovich08,DowZhou},
the construction of non-Tukey-top ultrafilters was addressed much later by
Milovich~\cite{Milovich08}, and Dobrinen and Todorcevic~\cite{Dobrinen/Todorcevic11},
and brought about the active subject of the Tukey order on ultrafilters over $\omega$.
They showed that consistently there are non-Tukey-top ultrafilters over $\omega$.
More precisely, Milovich constructed one from $\diamondsuit$, while Dobrinen and Todorcevic showed that a $p$-point over $\omega$ is non-Tukey-top.
In the last decade, the subject has been studied intensively by Dobrinen, Raghavan, Shelah,
Todorcevic, and others~\cite{TodorcevicDirSets85,Dobrinen/Todorcevic11,Raghavan/Shelah17,Raghavan/Todorcevic12};
for a survey on the matter see~\cite{DobrinenTukeySurvey15}.
Recently, Cancino and Zaplatal~\cite{CancinoZaplatal} announced the full resolution of Isbell's problem by showing that it is consistent that every nonprincipal ultrafilter over $\omega$ is Tukey-top.

As with Kunen's problem, it is natural to generalize Isbell's problem to other cardinals.
It is easily seen that a positive answer to Isbell's problem at $\kappa$ (in ZFC) implies a negative answer to Kunen's problem:
if $[\lambda]^{<\omega} \leq_T \Ucal$, then
$\Ucal$ has character at least $\lambda$.
Also, since every uniform ultrafilter on a regular cardinal $\kappa$ has character at least $\kappa^+$,
$2^\kappa = \kappa^+$ implies that all uniform ultrafilters over $\kappa$ have character exactly $2^\kappa$.
On the other hand, the equality $2^\kappa = \kappa^+$ does not trivialize Isbell's problem at $\kappa$ in the same way.
For instance, while the Proper Forcing Axiom (PFA) implies $2^{\aleph_1} = \aleph_2$,
it is not known if PFA implies that every uniform ultrafilter over $\omega_1$ is Tukey-top.

We extend this study and consider the Tukey order of ultrafilters over uncountable cardinals.
We establish a full independence result for Isbell's question on $\omega_1$.
For the first half of this result we prove:
\begin{theorem*}
    It is consistent that every uniform ultrafilter over $\omega_1$ is Tukey-top.
\end{theorem*}
\noindent
We present several  models for this:
\begin{enumerate}
    \item The usual forcing extension adding $2^{\aleph_1}$-many Cohen or random reals.
    \item The Cancino-Zaplatal model where every ultrafilter over $\omega$ is Tukey-top.
    \item A model due to P{\v r}ikr\'y where GCH holds.
    \item The Abraham-Shelah model~\cite{AbrahamShelah1986} and its generalization for successors of singular cardinals~\cite{bgp}.
\end{enumerate}
Note that after adding $\omega_2$-many Cohen reals (via $\Add(\omega, \omega_2)$) to a model of GCH, there is a non-Tukey-top ultrafilter over $\omega$, since $\mathfrak{d}=\mathfrak{c}$ in that model and Ketonen~\cite{Ketonen1976} proved that this is sufficient for there to be a $p$-point.

The other half of the independence of Isbell's question at $\omega_1$ is obtained via
a classical construction due to Laver~\cite{LAVER1982297} of a uniform ultrafilter over $\omega_1$ which is $\omega_1$-generated modulo
a countably complete ideal over $\omega_1$.
It is consistent relative to large cardinals that this construction can be carried out.
We will also leverage work of Galvin to show that the constructed ultrafilter exhibits even stronger combinatorial properties.
More generally, we establish that weakly normal ultrafilters are not Tukey-top.
Several other notions and constructions will be addressed, relevant to the 
extraordinary work from the 1970's on non-regular ultrafilters over $\omega_1$ \cite{KanamoriWeakly,Kanamori1978,KetonenBenda,TAYLOR197933}.
Recently, Usuba~\cite{Usuba} used related ideas to address questions about the monotonicity of the ultrafilter number.
In section~\ref{Section: Usuba} we show that this investigation is more general, and in fact yields comparisons of the Tukey types
of ultrafilters over different cardinals.

The Tukey order on uniform ultrafilters over measurable cardinals was recently studied by Benhamou and Dobrinen~\cite{TomNatasha}.
Many results from the Tukey order of ultrafilters over $\omega$ generalize to measurable cardinals, but also some fundamental differences appear.
For example, over a measurable cardinal $\kappa$ there is always a non-Tukey-top ultrafilter, and in fact a $\kappa$-complete non-$\kappa$-Tukey-top ultrafilter
(see definition~\ref{Def: Tukey top}).
This is because a measurable cardinal $\kappa$ always carries a normal ultrafilter, which is necessarily
non-$\kappa$-Tukey-top.
Moreover, in contrast to Isbell's result on $\omega$, $\kappa$-complete $\kappa$-Tukey-top ultrafilters might not exist; 
for example, Benhamou and Gitik~\cite{Parttwo} noticed that in Kunen's
$L[U]$, where $U$ is a the normal measure, there is no $\kappa$-Tukey-top $\kappa$-complete ultrafilter over $\kappa$. This was later generalized by Benhamou~\cite{SatInCan} and
Benhamou-Goldberg~\cite{TomGabe24} to other canonical inner models. 
In section~\ref{Sction: Tukey-top ultrafilters ZFC} we provide another construction for Tukey-top ultrafilters which extend the club filter over any cardinal $\kappa$ of uncountable cofinality.
For this, we introduce the notion of stationarily-independent families and show that such families exist in ZFC for any cardinal of uncountable cofinality.
This gives an answer to~\cite[Q. 5.4]{SatInCan}, and improves the construction from~\cite{TomNatasha}.

In the remaining part of this paper, we analyze Todorcevic's ultrafilter $\Ucal(T)$ using fragments of the PFA.
This ultrafilter is defined for a coherent Aronszajn tree (A-tree) $T$ on $\omega_1$ and in general yields a uniform filter $\Ucal(T)$.
Moreover, if the class of c.c.c. forcings is closed under taking products (a consequence of $\MA_{\omega_1}$),
$\Ucal(T)$ is an ultrafilter~\cite{l-maps}.

We show that PFA implies that $\Ucal(T)$ is Tukey-top and also minimal in the Rudin-Keisler order among uniform ultrafilters over $\omega_1$.
This complements previous work of Todorcevic~\cite{l-maps,UT_selective}.

\begin{theorem*}
Assume $\PFA(\omega_1)$.
For any coherent A-tree $T$, $[\omega_2]^{<\omega} \leq_T \Ucal(T)$.
In particular, PFA implies $\Ucal(T)$ is Tukey-top.
\end{theorem*}

\begin{theorem*}
Assume $\PFA(\omega_1)$.
If $T$ is any coherent A-tree and $f \mathrel{:} \omega_1 \to \omega_1$, then there is a $U \in \Ucal(T)$ such that $f \restriction U$ is either bounded or one-to-one.
\end{theorem*}

Combining this with work of Todorcevic~\cite{UT_selective} yields the following corollary.
\begin{cor*}
Assume $\PFA(\omega_1)$.
If $T$ is any coherent A-tree and $f \mathrel{:} \omega_1 \to \omega$ is any function which is not constant on a set in $\Ucal(T)$, then $f$ is a finest partition with respect to $\Ucal(T)$.
\end{cor*}    

It is not hard to show that a uniform ultrafilter over $\omega_1$ is weakly normal if and only if it both extends the club filter and is $\leq_{RK}$-minimal with respect to being a uniform ultrafilter over $\omega_1$.
While Laver has shown that $\MA_{\omega_1}$ implies there
are no weakly normal ultrafilters over $\omega_1$~\cite{LAVER1982297},
we show that this result does not decide whether $\Ucal(T)$ 
extends the club filter.

\begin{theorem*}
    It is relatively consistent with $\MA_{\omega_1}$ that there is a coherent A-tree $T$ such that $\Ucal(T)$ extends the club filter.
\end{theorem*}

This paper is organized as follows.
In section~\ref{Section: Preliminaries} we provide the basics of the relevant theory of the Tukey order and previously known results about the Tukey types of ultrafilters over uncountable cardinals.
Section~\ref{Seciton: destinguised} reviews and establishes some basic facts about certain benchmarks
in the Tukey order which will be needed later in the paper.
We establish the independence of Isbell's question on $\omega_1$ in section~\ref{section: Isbell} and
establish the consistency of every ultrafilter over $\omega_1$ is Tukey-top, and in section~\ref{Section: Non-Tukeytop} we settle Isbell's problem at higher cardinals.
In section~\ref{Seciton: PFA} we prove our results about $\Ucal(T)$.
We close with a list of questions and possible future directions in section~\ref{Section: Questions}.

\section{Preliminaries and Basic Results}\label{Section: Preliminaries}

We now fix some notational conventions and review some of the standard terminology which we will use throughout the rest of the paper.

Throughout much of the paper, we assume the reader has a background in modern set theory.
The texts~\cite{Jech2003} and \cite{Kunen} provide a broad foundation in set theory;
\cite{kanamori1994} covers large cardinals and related concepts.
Information on $\MA_{\omega_1}$ can be found in~\cite{Kunen};
information on PFA can be found in~\cite{AbrahamHandbook} and \cite{notes_FA}.
 
If $f$ is a function and $A$ is a subset of the domain of $f$, we will use $f[A]$ to denote the
image of $A$ under $f$.
For a set $A$, and a
cardinal $\lambda$, $[A]^\lambda$ denotes the
collection of subsets of $A$ of cardinality
$\lambda$. Similarly, $[A]^{<\lambda}$ is the
collection of subsets of $A$ of cardinality
$<\lambda$. The set $A^\lambda$ denotes the set of
all functions $f\mathrel{:}\lambda\to A$ and $A^{<\lambda}$ 
denotes the set of functions of the form
$f\mathrel{:}\alpha\to A$ for some $\alpha<\lambda$.
For $f,g\in \lambda^\kappa$, any binary relation
$R$ on $\lambda$, and any ideal $I$ over $\kappa$
we write $g \mathrel{R_I} f$ if and only if
$\{\alpha<\kappa\mid g(\alpha) \mathrel{\centernot{R}} f(\alpha)\}\in I$.
In particular, we write $f\leq g$ when for every $\alpha$, $f(\alpha)\leq g(\alpha)$ and $f\leq^* g$ if $g\leq_{J^\kappa_{bd}}f$, where $J_{bd}^\kappa=\{X\subseteq\kappa\mid \sup(X)<\kappa\}$ is the bounded ideal over $\kappa$.
These relations are typically not antisymmetric but induce antisymmetric relations on the associated
equivalence classes.
We will often abuse notation by working with representatives rather than equivalences classes even though
we will treat these as partial orders.
We denote by $\Add(\mu,\lambda)$ the Cohen forcing consisting of partial functions $f \mathrel{:} \mu \times \lambda\to 2$ such that $|f|<\mu$. 

    \subsection{The Tukey order}

Recall that a \emph{poset} is a set $P$ equipped
with a transitive, reflexive, antisymmetric
relation $\leq$.
A poset is (upward) \emph{directed} if for any $p, q \in P$ there is $r \in P$ with $r \ge p, q$.
A \emph{directed set} is a poset which is directed.
For $A, B \subseteq P$, we write $A \le B$
when for every $a \in A$ and $b \in B$, $a \le b$.
If an element of $P$ appears in a relation with a
set, the meaning is to replace it with its
singleton (e.g. $p \le A$ means $\{ p \} \le A$).

For $\mu$ a cardinal, a poset $(P,\leq)$ is called
\emph{$\mu$-directed} when for any $A \subseteq P$
with $|A| < \mu$ there is $p \in P$ with $p \ge A$.
Note that directed is the same as $\omega$-directed.

A subset $A$ of a poset $(P,\leq)$ is:
        \begin{enumerate}
            \item \emph{bounded} if there is $p \in P$ such that $A \le p$,
            \item \emph{cofinal} if for every $p \in P$ there is $a \in A$ such that $p \le a$.
        \end{enumerate}
        The \emph{cofinality} of a poset $P$, denoted $\cf \, P$, is the minimum cardinality of a cofinal subset. 

    \begin{definition}
        Let $(P, \le_P)$,$(Q, \le_Q)$ be posets.
        A function $f \mathrel{:} P \rightarrow Q$ is
        \begin{enumerate}
            \item \emph{monotone} if whenever $p, q \in P$ and $p \le_P q$, $f(p) \le_Q f(q)$,
            \item \emph{Tukey} if for every bounded $B \subseteq Q$, $f^{-1}(B)$ is bounded in $P$.
            \item \emph{cofinal} if for every cofinal $A \subseteq P$, $f[A]$ is cofinal in $Q$.
        \end{enumerate}
        The poset $P$ is \emph{Tukey-reducible} to $Q$, written $P \le_T Q$, if
        there is a Tukey map $f \mathrel{:} P \rightarrow Q$, or equivalently if there is a cofinal map $g \mathrel{:} Q \rightarrow P$.
    \end{definition}
    It is immediate from the definitions that if $(P,\leq_P)\leq_T (Q,\leq_Q)$, then $\cf(P,\leq_P)\leq \cf(Q,\leq_Q)$.

    The equivalence classes of the Tukey reducibility order are called \emph{Tukey types}.

    For any infinite cardinal, Tukey proved that
    there is a $\leq_T$-maximum directed set of cardinality
    $\kappa$.

    \begin{proposition}[{\cite[Thm. 5.1]{Tukey40}}]
        For any directed set $(P, \le_P)$ such that $|P|\leq\kappa$, there is a Tukey reduction $(P, \le_P) \le_T ([\kappa]^{<\omega}, \subseteq)$.
    \end{proposition}

    \begin{definition}\label{Def: Tukey top}
    A directed set $P$ is \emph{$(\mu,\lambda)$-Tukey-top} if
    there exists a collection $A \in [P]^{\lambda}$ such that every $B\in [A]^{\mu}$ is unbounded in $P$.
    In the context of ultrafilters over a cardinal $\kappa$, by \emph{$\mu$-Tukey-top} we mean $(\mu, 2^\kappa)$-Tukey-top.
    \emph{Tukey-top} means $\omega$-Tukey-top, as clarified in the following theorem.
    \end{definition}

    \begin{theorem}[Tukey~\cite{Tukey40}]\label{Thm: Tukey Combi Characterization of Tukey Top} 
        Let $\lambda$ and $\mu$ be regular cardinals, and suppose that $\cf([\lambda]^{<\mu},\subseteq)=\lambda$.
        The following are equivalent for any poset $P$:
        \begin{enumerate}
            \item $P$ is $(\mu,\lambda)$-Tukey-top.
            \item $[\lambda]^{<\mu}\leq_T P$.
        \end{enumerate}
    \end{theorem}
    
    \begin{theorem}[Schmidt {\cite[Thm. 14]{Schmidt55}}] \label{Thm: Schmidt maximality of the Tukey sets}
    If a $\mu$-directed poset $P$ has cofinality $\lambda$ then $P \le_T [\lambda]^{<\mu}$.
    \end{theorem}

    \begin{fact}[Folklore]\label{downwards closed directed}
    If $P \le_T Q$ and $Q$ is $\mu$-directed then $P$ is $\mu$-directed.
    \end{fact}

    \begin{proof}
    Suppose for contradiction that $A \in [P]^\lambda$ has no upper bound, where $\lambda < \mu$. 
    Let $\kappa = |P|$.
    Since $P \leq_T [\kappa]^{<\mu}$, there is an unbounded map $f \mathrel{:} P\rightarrow [\kappa]^{<\mu}$, which can easily be made injective (say by reserving $\kappa$ elements of $\kappa$ to serve as labels).
    Hence $f[A]$ is also unbounded.
    This is a contradiction, since $[\kappa]^{<\mu}$ is $\mu$-directed.
    \end{proof}

    \subsection{Ultrafilters}

    Recall that $\Fcal$ is a \emph{filter}
    over a set $X$ if $\Fcal \subseteq \mathscr{P}(X)$
    is nonempty, upwards closed, downwards directed,
    and does not contain $\emptyset$.
    A filter $\Ucal$ is an \emph{ultrafilter} if it is
    maximal under inclusion with respect to being a filter or,
    equivalently, for every $Y \subseteq X$
    either $Y$ or $X \setminus Y$ is in $\Ucal$.
    An ultrafilter is \emph{nonprincipal}
   if it does not contain any singletons.
   It is \emph{uniform} if all sets in the ultrafilter have the same cardinality.

    In this paper we shall primarily be concerned with Tukey types of uniform ultrafilters, considered as directed posets under reverse inclusion.
   The next lemma is useful when comparing
   ultrafilters using the Tukey order.
   
    \begin{lemma}[{\cite[Fact 6]{Dobrinen/Todorcevic11}}]
        If $\mathcal U$, $\mathcal V$ are uniform ultrafilters over an infinite cardinal $\kappa$ and $\mathcal U \le_T \mathcal V$, then there is a monotone cofinal map
        $f \mathcal{:} \mathcal V \rightarrow \mathcal U$.
    \end{lemma}
\noindent
    Thus Tukey reductions between uniform ultrafilters over the same set are always witnessed by monotone cofinal maps.

While it will be more tangential to our discussion,
the most important order on ultrafilters is a further refinement of the Tukey order known as 
the \emph{Rudin-Keisler order}.
    \begin{definition}
        Let $\mathcal U$ be an ultrafilter over a set $X$, and let $f \mathrel{:} X \rightarrow Y$ be a function.
        The projection $f_* \mathcal U$ of $\mathcal U$ to $Y$ along $f$ is the ultrafilter
        \[ \{ B \subseteq Y \mid f^{-1}(B) \in \mathcal U \}.\]
        For $\mathcal V$ an ultrafilter over $Y$, we say that $\mathcal V$ is \emph{Rudin-Keisler reducible} to $\mathcal U$ and write $\mathcal V \le_{RK} \mathcal U$ when there is $f \mathrel{:} X \rightarrow Y$ such that $\mathcal V = f_* \mathcal U$.
    \end{definition}

    It is a straightforward consequence of the definitions that for ultrafilters $\mathcal U$ and $\mathcal V$, if $\mathcal U \le_{RK} \mathcal V$ then $\mathcal U \le_T \mathcal V$.
    Ultrafilters $\mathcal U$ and $\mathcal V$ are said to be \emph{isomorphic} when there is a bijection $f$ between their underlying sets such that $\mathcal V = f_* \mathcal U$, and it is known that if $\mathcal U \le_{RK} \mathcal V$ and $\mathcal V \le_{RK} \mathcal U$ then $\mathcal U$ and $\mathcal V$ are in fact isomorphic.

We will pause here to remark that ultrafilters appear in many different contexts in the literature
and tend to be denoted in many different ways: by $p$ and $q$ in the study of the \v{C}ech-Stone compactification of $\omega$ to emphasize their role as points; by $\Ucal$ and $\Vcal$ for ultrafilters over $\omega$ or other small cardinals when one wishes to emphasize that they are collections of sets; by $U$ and $V$ in the context of large cardinals.
As different parts of this paper are closest to the contexts of these different notational traditions, our conventions will shift.
This should cause no confusion; nonetheless, we alert the reader to promote clarity.

    We are interested in Tukey types of uniform ultrafilters over regular uncountable cardinals, with $\omega_1$ as the most salient cardinal and the central questions being whether all ultrafilters over a given cardinal are $(\mu,\lambda)$-Tukey-top for fixed regular cardinals $\mu\le\lambda$. The study of such ultrafilters traces back to Keisler~\cite{KeislerChang}, who introduced the following notion motivated from a model-theoretic point of view:

    \begin{definition}
        Let $\lambda \le \mu$ be cardinals.
        An ultrafilter $\mathcal U$ is \emph{$(\lambda,\mu)$-regular} if there is a set $\mathcal A \subseteq \mathcal U$ such that $|\mathcal A| = \mu$ and for every $\mathcal B \subseteq A$ with $|\mathcal B| = \lambda$, $\bigcap \mathcal B = \emptyset$. If $\mathcal{U}$ is a uniform ultrafilter over $\kappa$ we say that $\mathcal{U}$ is \textit{regular} if it is $(\omega,\kappa)$-regular.
    \end{definition}
    Regularity-like properties were later studied  in the 1970s in a series of influential papers by Ketonen-Benda~\cite{KetonenBenda}, Kanamori~\cite{KanamoriWeakNormal,finest_part,Kanamori1976UltrafiltersOA,Kanamori1978}, Kunen~\cite{Kunen1972} and Taylor~\cite{TAYLOR197933}. The following definition is highly connected to regularity:

    \begin{definition}
    A uniform ultrafilter $\Ucal$ over $\kappa$ is called \emph{weakly normal} if for any regressive function $f\mathrel{:}\kappa\rightarrow \kappa$, there is $\theta<\kappa$ such that $f^{-1}[\theta]\in \Ucal$. Equivalently, $[id]_\Ucal=\sup_{\theta<\kappa}[c_\theta]_\Ucal$.  
    \end{definition}
    It is well-known that weakly normal ultrafilters extend the club filter. The next theorem establishes the equivalence of the existence of non-regular ultrafilters with the existence of weakly normal ones:
\begin{theorem}[Kanamori \cite{KanamoriWeakNormal}, Ketonen-Benda \cite{KetonenBenda}] Let $\mathcal{U}$ be a uniform ultrafilter over $\kappa^+$, then: 
    \begin{enumerate}
        \item If $\Ucal$ is weakly normal then $\Ucal$ is non-regular.
        \item If $\Ucal$ is non-regular, then $\Ucal$ is above a weakly normal ultrafilter in the Rudin-Keisler order.
    \end{enumerate}
\end{theorem}

\begin{theorem}[Laver \cite{LAVER1982297}]\label{Thm: MA implies every ultrafilter is regular}
$MA_{\omega_{1}}$ implies that every uniform ultrafilter over $\omega_1$ is regular.
In particular, $\MA_{\omega_1}$ implies there are no weakly normal ultrafilters over $\omega_1$.
\end{theorem}

Kanamori~\cite{Kanamori1978} studied $(\mu,\lambda)$-Tukey-top ultrafilters, though under a different name, as a weakening of regular ultrafilters, proving that any uniform ultrafilter over an uncountable cardinal is $(2^\kappa,2^\kappa)$-Tukey-top. Shortly after, Taylor and Galvin proved the following results which constitute a starting point for the investigation of this paper:
\begin{theorem}[Taylor {\cite[Thm 2.4(2)]{TAYLOR197933}}]
If $U$ is a uniform ultrafilter over a successor cardinal $\kappa^+$, then $U$ is $(\kappa,\kappa^+)$-Tukey-top.
\end{theorem}

\begin{theorem}[{Galvin; appears as~\cite[Thm. 3.3]{baumgartner-hajnal-mate}}]\label{Thm: Galvin}
    Let $\mu$ be a cardinal such that $\mu^{<\mu} = \mu$.
    Then for any normal filter $\mathcal F$ over $\mu$, $\mathcal F$ is not $(\mu, \mu^+)$-Tukey-top.
\end{theorem}

   Although we will mostly be interested in small uncountable cardinals, let us mention that recently the topic of (non)-$(\mu,\lambda)$-Tukey-top ultrafilters gained renewed interest in the case of measurable cardinals under yet another name---the Galvin property---due to several new applications (e.g.~\cite{TomNatasha,bgp,Parttwo,TomGabe24}).

\section{Distinguished Tukey-types Related to Ultrafilters over Uncountable Cardinals}\label{Seciton: destinguised}
In this section, we study some Tukey types which relate to ultrafilters over a regular uncountable cardinal $\kappa$, namely cofinal types of cardinality at most $2^{\kappa}$. 

\subsection{The Abraham-Shelah model and the Tukey-type of the club filter}
Let us denote the club filter by $\text{Cub}_\kappa$; this is the filter generated by clubs\footnote{i.e. set which are closed in the order topology of $\kappa$ and unbounded.} in $\kappa$. First observe that the Tukey type of the club filter is the following:
\begin{lemma}[Folklore]
For $\kappa$ regular, $(\text{Cub}_{\kappa},\supseteq)\equiv_T(\kappa^{\kappa},\leq)$ 
\end{lemma}
\begin{proof}
    $\text{Cub}_{\kappa}\leq _T \kappa^{\kappa}$ is witnessed by the unbounded function $X\mapsto f_X$, where $f_X$ is the increasing enumeration of the club $X$. The other direction is witnessed by the cofinal function $f\mathrel{:}\text{Cub}_{\kappa}\rightarrow \kappa^{\kappa}$ defined by $f(X)(\alpha)=f_X(\alpha+1)$. Certainly both of these functions are monotone, and one checks that they are unbounded and cofinal, respectively, by examining clubs of closure points of elements of $\kappa^\kappa$.
\end{proof}

It follows that the generalized dominating number $\mathfrak{d}_\kappa=\cf(\kappa^\kappa,\leq^*)=\cf(\kappa^\kappa,\leq)$ is just $\chi(\text{Cub}_\kappa)$; as usual, $\mathfrak{d}$ denotes $\cf (\omega^\omega,\leq)$.
The fact that the club filter is $\sigma$-complete automatically rules out the possibility of it being Tukey-top, but can it be $\omega_1$-Tukey-top?
As stated in the previous section (see theorem~\ref{Thm: Galvin}), Galvin showed that under $\kappa^{<\kappa}=\kappa$, every normal filter over $\kappa$ is not Tukey-above $[\kappa^+]^{<\kappa}$.
Abraham-Shelah proved the following theorem~\cite{AbrahamShelah1986}:
\begin{theorem*}[Abraham-Shelah Model]
\label{thmabsh} Assume $\mathsf{GCH}$ holds. Suppose $\kappa,\lambda$ are infinite cardinals such that $\cf(\kappa)=\kappa<\kappa^+<\cf(\lambda)\leq\lambda$. 
Then, in a forcing extension there is a family $\mathcal{C}$ of $\lambda$-many clubs in $\kappa^+$, such that:
\begin{center}
    For every subfamily $\mathcal{D}\subseteq \mathcal{C}$ with $|\mathcal{D}|=\kappa^+$, $|\bigcap \mathcal{D}|<\kappa$.
\end{center}
Moreover,  $2^\kappa=2^{\kappa^+}=\lambda$ holds in this model provided $\cf(\lambda)>\kappa^+$.
\end{theorem*}
In particular, in the Abraham-Shelah model, the club filter $\text{Cub}_{\kappa^+}$ is $\kappa^+$-Tukey-top. At the successor of a singular cardinal, this was established in~\cite{bgp}, and for a (weakly) inaccessible cardinal $\kappa$, it is still open whether the club filter can be $\kappa$-Tukey-top~\cite[Q. 5.7]{bgp}.
Applying this theorem to $\omega_1$, we can find a model with $2^{\aleph_1}$-many clubs such that the intersection of any $\aleph_1$-many of these clubs is finite. It follows that any extension of the club filter is not Tukey-top. More generally, the club filter enjoys the property of being \textit{deterministic}, as introduced in~\cite{Commutativity}.
A filter $\mathcal{F}$ is \emph{deterministic} if it is generated by a set $\mathcal{B}$ such that for any $\mathcal{A}\subseteq \mathcal{B}$, if $\bigcap \mathcal{A}\notin \mathcal{F}$ then $\bigcap\mathcal{A}\in \mathcal{F}^*$. Deterministic filters have the property that if $F \subseteq F'$ then $F\leq_T F'$.

\begin{proposition}
    $\text{Cub}_{\kappa}$ is a deterministic filter. Hence, any uniform extension of the club filter is Tukey-above it.
\end{proposition}
 So in the Abraham-Shelah model, in fact any extension of $\text{Cub}_{\omega_1}$ is $\omega_1$-Tukey-top.  
In the next section, we will moreover see that in this model every uniform ultrafilter over $\omega_1$ is Tukey-top.

\subsection{On the cofinal type of \texorpdfstring{$\kappa^{(\kappa^+)}$}{k(k+)}}
\label{Section: omegaomega1}  

The directed set $(\omega^{\omega_1},\leq)$ will play an important role later in the paper.
Much of the basic analysis we will need readily generalizes to higher cardinals.

\begin{lemma}\label{lemgentoeverywhere}
    Suppose that $I$ is a $\kappa^+$-complete ideal over $\kappa^+$. Then $$(\kappa^{(\kappa^+)},\leq_I)\equiv_T (\kappa^{(\kappa^+)},\leq).$$
\end{lemma}
\begin{proof}
      Clearly, the identity map is a Tukey reduction witnessing that $$(\kappa^{(\kappa^+)},\leq_I) \leq_T (\kappa^{(\kappa^+)},\leq).$$
      For the other direction, the assumption that $I$ is $\kappa^+$-complete ensures the existence of a partition $\kappa^+=\biguplus_{i<\kappa^+}A_i$ where each $A_i\in I^+$ (see~\cite[16.3]{kanamori1994}).  Consider the map $f\mapsto F(f)$, where
    $F(f)(\alpha):=f(i)$ for the unique $i<\kappa^+$ such that $\alpha\in A_i$. To see that $F$ is a Tukey reduction, let $\mathcal{A}\subseteq \kappa^{(\kappa^+)}$ be unbounded in $\leq$, meaning there is $i^*<\kappa^+$ such that $\{f(i^*)\mid f\in\mathcal{A}\}$ is unbounded in $\kappa$. Suppose for contradiction that $g$ is a $\leq_I$-bound for $F[\mathcal{A}]$. Consider $g\restriction A_{i^*}$, and note that since $I$ is $\kappa^+$-complete, there must be a positive $A'\subseteq A_{i^*}$ such that $g\restriction A'$ is constantly $\alpha^*$ for some $\alpha^* < \kappa$. This is impossible since for every $f\in\mathcal{A}$, $F(f)\leq_I g$, so there is $\alpha_f\in A'$ such that $\alpha^*=g(\alpha_f)\geq F(f)(\alpha_f)= f(i^*)$, in which case $\alpha^*$ bounds $\{f(i^*)\mid f\in \mathcal{A}\}$ within $\kappa$.
\end{proof}
The following lemma provides a significant lower bound for the Tukey type of $\kappa^{(\kappa^+)}$.
\begin{lemma}
Suppose that there is a family $\mathcal{F} \subseteq [\kappa^+]^{\kappa^+}$ which is almost disjoint modulo bounded.
Then $[\lambda]^{<\kappa} \leq_T \kappa^{(\kappa^+)}$, where $\lambda = \lvert \mathcal F \rvert$.
\end{lemma}

\begin{proof}
    Fix injections $e_\beta\mathrel{:}\beta \to \kappa$ for each $\beta < \kappa^+$, and for $X \in \mathcal{F}$
    define $f_X\mathrel{:}[\kappa^+]^2 \to \kappa$ by 
    \[
    f_X(\alpha,\beta) :=
    \begin{cases}
    e_\beta(\min(X \cap [\alpha,\beta)) & \textrm{ if } X \cap [\alpha,\beta) \ne \emptyset \\
    0 & \textrm{ otherwise}
    \end{cases}
    \]
    Since $\kappa^{[\kappa^+]^2} \equiv_T \kappa^{\kappa^+}$, it suffices to show that $X \mapsto f_X$ has the property that whenever
    $\mathcal{F}_0 \subseteq \Fcal$ has cardinality $\kappa$, $\{f_X \mid X \in \mathcal{F}_0\}$ is unbounded.
Let $\mathcal{F}_0 \subseteq \mathcal{F}$ be a set of cardinality $\kappa$.
Let $\alpha<\kappa^+$ be sufficiently large that $\{X \setminus \alpha \mid X \in \mathcal{F}_0\}$
is pairwise disjoint and let $\beta > \alpha$ be such that $\min(X \setminus \alpha) < \beta$ for all $X \in \mathcal{F}_0$.
Then $X \mapsto f_X(\alpha,\beta)$ is one-to-one and hence $\{f_X \mid X \in \mathcal{F}_0\}$ is unbounded in
$\kappa^{[\kappa^+]^2}$.
\end{proof}
\begin{remark}
 Any family $\mathcal{F}\subseteq \kappa^{(\kappa^+)}$ of functions different modulo bounded\footnote{i.e. for any distinct $f,g\in\mathcal{F}$, $\{\alpha\mid f(\alpha)=g(\alpha)\}$ is bounded in $\kappa$.} induces a family of subsets of $\kappa^+$ almost disjoint modulo bounded of the same cardinality. This is proven by transferring the graphs of the functions through a bijection of $\kappa^+\times\kappa^+$ with $\kappa^+$. The other direction is also clear: any family of functions different modulo bounded sets induces a family of almost disjoint functions modulo bounded of the same cardinality.
\end{remark}
Recall that the generalized bounding number, denoted by $\mathfrak{b}_\mu$, is the minimal size of an unbounded family in $(\mu^\mu,\leq^*)$.  It is well known that $\mu^+\leq\mathfrak{b}_{\mu}\leq 2^{\mu}$.
\begin{corollary}\label{Cor: omegaomega1 lowerbounds}
    \ {}\begin{enumerate}
        \item \label{b_kappa+_Tukey} $[\mathfrak{b}_{\kappa^+}]^{<\omega}\leq_T \kappa^{(\kappa^+)}$. 
        \item  \label{subsetcoding}$2^\kappa=\kappa^+$ implies that $\kappa^{(\kappa^+)}$ is Tukey-top.
    \end{enumerate} 
\end{corollary}
\begin{proof}
    Both items use the previous remark. For (\ref{b_kappa+_Tukey}), it is possible to recursively define a strictly increasing chain of length $\mathfrak{b}_{\kappa^+}$ of functions increasing modulo bounded. For (\ref{subsetcoding}), if $2^{\kappa}=\kappa^+$, then for every $\alpha<\kappa^+$, let $\pi_\alpha\mathrel{:} \mathscr{P}(\alpha)\to \kappa^+$ be an injection. For every $X\subseteq \kappa^+$ set $f_X(\alpha)=\pi_\alpha(X\cap\alpha)$. Now if $X\neq Y$ then there is $\alpha<\kappa^+$ such that for every $\beta\geq\alpha$, $X\cap \beta\neq Y\cap \beta$ and therefore $f_X(\beta)\neq f_Y(\beta)$.
\end{proof}
\begin{remark}
    It is impossible to prove in ZFC the existence of $2^{\kappa^+}$-many almost disjoint subsets of $\kappa^+$. Indeed, Baumgartner~\cite{BAUMGARTNER1976401} proved that consistently there is no such family. In that paper Baumgartner also gives more assumptions under which there are $2^{\kappa^{+}}$-many almost disjoint subsets of $\kappa^+$, and therefore additional assumptions guaranteeing $\kappa^{(\kappa^+)}$ is Tukey-top. 
\end{remark}

\begin{corollary}
 \ {} \begin{enumerate}
     \item \label{kappa^kappa+_Tukey}
     $(\kappa^+)^{\kappa^+} <_T \kappa^{(\kappa^+)}$.
     \item \label{kappa^kappa^kappa+_Tukey}
     $\kappa^{(\kappa^+)}\equiv_T (\kappa^\kappa)^{\kappa^+}$.
 \end{enumerate}
\end{corollary}
\begin{proof}
    For (\ref{kappa^kappa+_Tukey}), we have that $\kappa^{(\kappa^+)}\geq_T [\mathfrak{b}_{\kappa^+}]^{<\kappa}$ and therefore $\kappa^{(\kappa^+)}\geq_T \kappa$. It follows that $\kappa^{(\kappa^+)}\equiv_T (\kappa^{(\kappa^+)})^{\kappa^+}\geq_T (\kappa^+)^{\kappa^+}$. The strict inequality follows from the fact that $\kappa^{(\kappa^+)}$ is not $\kappa^+$-directed and fact~\ref{downwards closed directed}. 
    (\ref{kappa^kappa^kappa+_Tukey}) is straightforward. 
    \end{proof}
Hence for example, since $\mathfrak{b},\mathfrak{d}\leq_T (\omega^\omega,\leq)$, then $\omega^{\omega_1}\geq_T \mathfrak{b}^{\omega_1},\mathfrak{d}^{\omega_1}$.
We do not know whether $\omega^{\omega_1}$ (for example) is provably Tukey-top in ZFC (see question~\ref{Question: omegaomega1}).

We now turn our attention to the relation between ultrafilters over $\omega_1$ and the cofinal type of $\omega^{\omega_1}$. The following notation will be used throughout the remainder of the paper.
Fix a sequence $\vec{e}=\Seq{ e_\beta \mid \beta <\omega_1}$ such that each $e_\beta\mathrel{:}\beta \rightarrow \omega$ is one-to-one.
Define $$U^{\vec{e}}_{\alpha,n} =\{\beta \in \omega_1 \mid \beta \leq \alpha\text{ or }n\leq e_\beta(\alpha)\}$$
and for a partial function $f\mathrel{:}\omega_1\to \omega$, set $$U^{\vec{e}}_f = \bigcap_{\alpha \in \dom(f)} U_{\alpha,f(\alpha)}.$$
In what follows, there will be no risk of ambiguity and we will suppress the superscript $\vec{e}$.

\begin{lemma} \label{Ulam_UF}
    If $\Ucal$ is a uniform ultrafilter over $\omega_1$, then there is an $\alpha_0$ such that for all $\alpha \geq \alpha_0$ and all $n \in \omega$,
    $U_{\alpha,n} \in \Ucal$. 
\end{lemma}
\begin{proof}
    Let $X$ be the set of all $\alpha<\omega_1$ such that for some $n$,
    $U_{\alpha,n}\notin \mathcal{U}$.  Define $g\mathrel{:}X \rightarrow \omega$ by $g(\alpha) = n$ if for 
    $\mathcal{U}$-many $\beta$'s, $e_\beta(\alpha) = n$.  This is always defined since, given $\alpha \in X$ and $n$ with $U_{\alpha,n} \not \in \Ucal$, $$\omega_1\setminus U_{\alpha,n} = \bigcup_{k \leq n} \{\beta > \alpha \mid e_\beta(\alpha) = k\}$$ and hence one of these sets is in $\mathcal{U}$.  One can also easily check that $g\mathrel{:}X \rightarrow \omega$ is one-to-one, and hence $X$ is countable.
\end{proof}

Let $I$ be an ideal over $\omega_1$ and $\mathcal{U}$ an ultrafilter over $\omega_1$. Consider the statement:
\begin{equation} \tag{$(\dagger)_{I,\Ucal}$} \forall f\mathrel{:}\omega_1\to\omega \, \exists X\in I^* \ \ U_{f\restriction X}\in \Ucal\end{equation}
This principle gives a sufficient condition for $\mathcal{U}$ to be above $\omega^{\omega_1}$ and will be used in later parts of the paper.
\begin{proposition}\label{prop: dagger implies above omegaomega1}
    Let $I$ be $\sigma$-complete. Then $(\dagger)_{I,\Ucal}$ implies $(\omega^{\omega_1},\leq) \leq_T \Ucal$. 
\end{proposition}
\begin{proof}
Arguing as in lemma~\ref{lemgentoeverywhere}, we fix $g\mathrel{:}\omega_1\to \omega_1$ such that $g^{-1}[\{i\}]\in I^+$ for every $i<\omega_1$. Let us describe an unbounded map from $\omega^{\omega_1}$ to $\Ucal$.  For each $f\mathrel{:}\omega_1\to\omega$, let $X_{f}\in I^*$ be a set such that $U_{(f\circ g)\restriction X_f} \in \Ucal$, which exists by $(\dagger)_{I,\Ucal}$. Define $U_f:=U_{(f\circ g)\restriction X_f}$.
Let $\Fcal \subseteq \omega^{\omega_1}$ be $\leq$-unbounded. 
Again, by replacing $\Fcal$ with a countable subset if necessary, we may assume $\Fcal$ is countable. 
Let $\delta<\omega_1$ be such that $\{f(\delta)\mid f\in \mathcal{F}\}$ is unbounded, and consider $g^{-1}[\{\delta\}]\in I^+$. Set $X := \bigcap \{X_f \mid f \in \Fcal\}$. Since $I$ is $\sigma$-complete, there is $\delta^* \in X\cap g^{-1}[\{\delta\}]$, hence $\{f(g(\delta^*)) \mid f \in \Fcal\}$ is unbounded.
Suppose towards a contradiction that $\bigcap \{U_f \mid f \in \Fcal\}\in \Ucal$, then by uniformity of $\mathcal{U}$, it would contain some $\beta > \delta^*$.
But then $e_\beta(\delta^*) > (f\circ g)(\delta^*)$ for all $f \in \Fcal$, contrary to our choice of $\delta^*$.
Thus $\bigcap \{U_f \mid f \in \Fcal\}\in \mathcal{U}$.

\end{proof}
Using corollary~\ref{Cor: omegaomega1 lowerbounds}, we immediately conclude:
\begin{corollary}\label{the: (PFA) Dagger implies tukey top}
    Assume $2^{\aleph_1} = \aleph_2$, and let $I$ be a $\sigma$-complete ideal. 
    Any uniform ultrafilter $\Vcal$ satisfying $(\dagger)_{I,\Vcal}$ is Tukey-top.
\end{corollary}

\section{Isbell's Question for Uncountable Cardinals}\label{section: Isbell}
In this section, we consider the analogue of Isbell's problem which was discussed in the introduction, concerning ultrafilters over uncountable cardinals. Perhaps surprisingly,
we will show that a positive answer to Isbell's problem on uncountable cardinals is witnessed by a fairly simple model, and the challenge seems to be concentrated on constructing models with non-Tukey-top ultrafilters. 

    \subsection{ZFC constructions}\label{Sction: Tukey-top ultrafilters ZFC}
    
   Isbell~\cite{Isbell65} in fact proved that there is a Tukey-top ultrafilter over every infinite cardinal.
\begin{proposition}[Isbell]
     For any infinite cardinal $\kappa$, there
     is a uniform ultrafilter $\Ucal$ over $\kappa$
     which is Tukey equivalent to $[2^\kappa]^{<\omega}$.
\end{proposition}
More precisely, Isbell constructed a maximal number of such ultrafilters using independent families.
 In~\cite[Prop. 3.21-3.22]{TomNatasha}, \textit{normal} $\kappa$-independent families (due to Hayut~\cite{filterextension}) were used to run a construction similar to Isbell's, resulting in Tukey-top ultrafilters which \textit{extend the club filter} (see theorem~\ref{Thm: tomnatasha} below).  Recall that $\Seq{ A_i\mid i<\lambda}$ is called a {\em normal} $\kappa$-independent family, if it is $\kappa$-independent and for any two disjoint subfamilies $\Seq{ A_{\alpha_i}\mid i<\kappa},\Seq{ A_{\beta_i}\mid i<\kappa}\subseteq \Seq{ A_i\mid i<\lambda}$, the diagonal intersection $\Delta_{i<\kappa}(A_{\alpha_i}\setminus A_{\beta_i})$ is a stationary subset of $\kappa$.

In contrast to standard $\kappa$-independent families, the existence of a normal $\kappa$-independent family is not guaranteed by ZFC alone. For example~\cite[Proposition 4.2]{SatInCan}
\label{normalfamily}, if $\diamondsuit(\kappa)$ holds, then there is a normal $\kappa$-independent family of length $2^\kappa$. 
\begin{theorem}[Benhamou-Dobrinen] \label{Thm: tomnatasha}
    Suppose that there is a normal $\kappa$-independent family of length $2^\kappa$, and let $\mu<\kappa$ be a cardinal. Then there is a $\mu$-complete filter 
    $F^1_{\mu,top}$ extending the club filter such that any extension of $F^1_{\mu,top}$ to an ultrafilter is $\mu$-Tukey-top. 
\end{theorem}
Hence, if $\diamondsuit(\kappa)$ holds, then there is a Tukey-top ultrafilter $\Ucal$ over $\kappa$ which extends the club filter. In particular, if $V=L$ then for every regular cardinal $\kappa$ there is an ultrafilter over $\kappa$ which is Tukey-top and extends the club filter.
Also if $\kappa$ is a strongly compact cardinal, then there is a $\kappa$-complete $\kappa$-Tukey-top ultrafilter extending the club filter.

Note that if $\kappa^{<\kappa}=\kappa$, then by Galvin's theorem~\ref{Thm: Galvin} the filter $F^1_{top}$ cannot be normal. Galvin's theorem emphasizes that the construction of $F^1_{top}$ uses heavier machinery than is needed, since normal independent families are primarily designed to give rise to normal filters. 
Let us provide another construction that removes the dependence on $\diamondsuit(\kappa)$. For this we introduce the following notions.
Given a sequence $\vec{X} =\Seq{ X_\alpha\mid \alpha<\lambda}$ of subsets of a cardinal $\kappa$, a \textit{flip} of $\vec{X}$ is a sequence of the form $\vec{X}^\sigma:=\Seq{ X^{\sigma(i)}_i\mid i\in \dom(\sigma)}$, where $\sigma$ is a (non-empty) partial function from $\lambda$ to $2$, and for every $i$,
$$X^{\epsilon}_i=\begin{cases}
        X_i & \epsilon=0\\
        \kappa\setminus X_i & \epsilon=1
    \end{cases}.$$
    If $\sigma\mathrel{:}\lambda\to 2$ is a total function, we say that $\vec{X}^\sigma$ is a \textit{full flip}. 
    In this context, it will be convenient to identify a sequence of sets with its 
    range. Thus for example we write $\bigcap \vec{X}^\sigma$ for $\bigcap\{ X_i^{\sigma(i)} \mid i \in \dom(\sigma)\}$.

\begin{definition}
    A family of sets $\Seq{ X_i\mid i<\lambda }\subseteq \mathscr{P}(\kappa)$ is called a \emph{$\mu$-stationary independent family} if any finite Boolean combination of length $<\mu$ of the family is stationary. That is, if for every partial $\sigma\mathrel{:}\lambda\to 2$ with $|\sigma|<\mu$, $\bigcap \vec{X}^\sigma$ is stationary in $\kappa$.
\end{definition}
A normal $\kappa$-independent family is a $\kappa$-stationary independent family~\cite{SatInCan}. Also note that any flip of a $\mu$-stationary independent family is $\mu$-stationary independent.
\begin{proposition}
    Let $\mu\leq\kappa$ and $\lambda\leq 2^\kappa$ be cardinals and suppose that there is a $\mu$-stationary independent family of subsets of $\kappa$ of size $\lambda$. Then there is a $\mu$-complete filter $F$ extending the club filter such that any extension of $F$ to an ultrafilter is $(\mu,\lambda)$-Tukey-top. 
\end{proposition}
\begin{proof}
    Let $\{X_i\mid i<\lambda\}$ be a $\mu$-stationary independent family of subsets of $\kappa$. Let $F$ be the $\mu$-complete filter generated by $\text{Cub}_\kappa\cup\{X_i\mid i<\lambda\}\cup\{\kappa\setminus (\bigcap_{i\in I}X_i)\mid I\in [\lambda]^\mu\}$.
    If we can show that $F$ is proper, then it clearly has the properties sought for. It remains to see that the generating set above has the finite intersection property. Let $I\in [\lambda]^{<\mu}$ and $\{I_\alpha\mid \alpha<\theta\}\subseteq [\lambda]^{\mu}$ where $\theta<\mu$. It suffices to show that $$ \bigcap_{i\in I}X_i\cap(\bigcap_{\alpha<\theta}(\kappa\setminus\bigcap_{j\in I_\alpha}X_j))$$ is stationary. Since all the $I_\alpha$'s have size $\mu$, we can find $j_\alpha\in I_\alpha\setminus I$ and simply note that the set $\bigcap_{i\in I}X_i\cap \bigcap_{\alpha<\theta}\kappa\setminus X_{j_\alpha}$ is a stationary subset of the above set.
\end{proof}
Finally, let us construct a $\mu$-stationary independent family of maximal size:
\begin{proposition}
Let $\kappa$ be a regular uncountable cardinal such that $\kappa^{<\mu}=\kappa$. Then there is $\mu$-stationary independent family of $2^\kappa$-many sets.\end{proposition}
\begin{proof}
     let $\Seq{ U_\xi \mid \xi < \kappa}$ be an enumeration of the clopen subsets of $2^\kappa$ in its $<\mu$-topology\footnote{Here we mean the usual product topology generated by the $<\mu$-supported product of the discrete topology on $2$.}, where each element is repeated stationarily often. Note that this enumeration is possible since $\kappa^{<\mu}=\kappa$. For any $x\in 2^\kappa$, define $I_x = \{\xi \in \kappa \mid x \in U_\xi\}$. We now argue that $\{I_x\mid x\in 2^\kappa\}$ is $\mu$-stationary independent. Consider any $I,J\in [2^\kappa]^{<\mu}$ such that $I\cap J=\emptyset$. Let $I\cup J=\{x_\alpha\mid \alpha<\theta\}$. Since there are less than $\mu$-many $x_\alpha$, there is a set $s\in[\kappa]^{\mu}$ such that $s\cap x_i\neq s\cap x_j$ for all $i\neq j<\theta$. Define $Y=\bigcup_{x\in I}B_{x,s}$, where $B_{c,s}$ is the basic clopen set of all $x$ such that $x\cap s=c$. By the choice of $s$, $x\in Y$ iff $x\in I$. Since $Y$ is indexed stationarily many times, for any $\xi$ such that $U_\xi=Y$,  $\xi\in I_{x}$ iff $x\in I$, and hence $\xi\in (\bigcap_{x\in I}I_x)\cap(\bigcap_{y\in J}I_y^c)$, as desired.
\end{proof}
\begin{corollary}
    Let $\kappa$ be a regular uncountable cardinal and $\mu$ be any cardinal such that $\kappa^{<\mu}=\kappa$. Then there are $2^{2^\kappa}$-many  $\mu$-Tukey-top ultrafilters over $\kappa$. 
\end{corollary}
 Since any regular cardinal satisfies $\kappa^{<\omega}=\kappa$, applying the above corollary to $\mu=\omega$  gives an answer to~\cite[Q. 5.4]{SatInCan}.
    Also, note that $\kappa^{<\mu}=\kappa$ is in fact equivalent to the existence of a $\mu$-independent family of length $\kappa$.

    \subsection{Consistency results}\label{Sec:EveryUltIsTukeyTop}

    Let us start by settling the consistency of an affirmative answer to Isbell's question of whether all uniform ultrafilters are Tukey-top in the case of cardinals $\kappa>\omega$. 
    We say that the sequence $\vec{X}$ of subsets of $\kappa$ has the \emph{flipping $\mu$-bounded intersection property}, if for any flip $\vec{X}^\sigma$ where $|\sigma|=\mu$,  $|\bigcap \vec{X}^\sigma|<\kappa$. 
    \begin{proposition}
         Suppose that there is a sequence of subsets of $\kappa$ of length $\lambda$ with the flipping $\mu$-bounded intersection property. Then every uniform ultrafilter $\Ucal$ over $\kappa$ is $(\mu,\lambda)$-Tukey-top.  
\end{proposition}
    \begin{proof}
        Given any ultrafilter $\Ucal$, there is a full flip $\vec{X}^\sigma$ such that $\vec{X}^\sigma\subseteq \Ucal$. The rest follows from uniformity.
    \end{proof}

\begin{theorem}\label{Thm:TukeyTop in Cohen}
    Let $\omega\leq\mu=\mu^{<\mu}<\kappa<\lambda$ be cardinals, $\mathbb{P}=\Add(\mu,\lambda)$, and $G\subseteq \mathbb{P}$ be any $V$-generic filter. Then in $V[G]$ there is a sequence $\Seq{ X_\alpha\mid \alpha<\lambda}\subseteq \mathscr{P}(\kappa)$ such that for every flip $\vec{X}^\sigma$ with $|\sigma|=\mu$, $|\bigcap_{i\in I}X^\sigma_i|\leq\mu$. In particular, if $cf(\kappa)>\mu$, then $\vec{X}$ has the flipping $\mu^+$-bounded intersection property. 
\end{theorem}

\begin{proof}
    Observe that $\Add(\mu,\lambda)$ is forcing equivalent to the poset of all partial functions $p\mathrel{:}\kappa \times \lambda \rightarrow 2$ such that $|\dom(p)| < \mu$.
    If $G$ is generic for this poset and $i \in \lambda$, define
    $$X_i = \{\alpha \in \omega_1 \mid \exists p \in G \ (p(i,\alpha) =1 )\}.$$ 
    Let us  prove that $\Seq{ X_i \mid i<\lambda}$ has the flipping $\mu$-bounded intersection property.
    Suppose towards a contradiction that $\sigma\mathrel{:}\lambda\to 2$ is a partial function, $|\sigma|=\mu$ and $|\bigcap \vec{X}^\sigma|\geq\mu^+$.
    Back in $V$, let $\dot{\sigma}$ be a name for $\sigma$ and use the $\mu^+$-chain condition, to find  $I\in V$, $|I|=\mu$ such that $\dom(\sigma)\subseteq I$. Let $p\in G$ be a condition forcing that $\dom(\dot{\sigma})\subseteq I$ and that $|\bigcap\dot{\vec{X}}^{\dot{\sigma}}|\geq \check{\mu}^+$.  
    Let $\theta$ be sufficiently large and $M\prec H_\theta$ be a model of size $\mu$,
    closed under $<\mu$-sequences, and such that $p,\dot{\sigma},\mathbb{P}, \vec{X},I,\mu \in M$,
    noting that $\mu+1\subseteq M$. It is easy to check that $\mathbb{P}\cap M=\Add(\mu,\lambda\cap M)$.
    Next, find  $\alpha \in \kappa \setminus M$  and $p' \leq p$  a condition such that $$p'\Vdash \check{\alpha}\in \bigcap\dot{\vec{X}}^{\dot{\sigma}}.$$
     Let $D\subseteq \Add(\mu,\lambda)$ be the set of all $q$ such that for some $i\in I\setminus \supp(p')$, and $\epsilon=0,1$, $q\Vdash \check{i}\in \dom(\dot{\sigma})\wedge \dot{\sigma}(i)=\epsilon$. Then $D$ is dense below $p$ as $\supp(p')$ has size $<\mu$ while $p$ forces $|\dot{\sigma}|=\check{\mu}$. Note that $D$ is definable in $M$ as $I\setminus \supp(p')=I\setminus (I\cap \supp(p'))$ and $I\cap \supp(p')\in M$ as $M$ is closed under ${<}\mu$-sequences. Fix a maximal antichain $\mathcal{A}\subseteq D$, such that $\mathcal{A}\in M$. By the chain condition, $\mathcal{A}\subseteq M$ and there is $p''\leq p$, $p''\in \mathcal{A}$ such that $p'$ and $p''$ are compatible. Fix $i,\epsilon$ witnessing that $p''\in D$. Since $i\notin \supp(p')$
     and $\alpha\notin M$, $\Seq{ i,\alpha}\notin \dom(p'\cup p'')$.  Define $q=p'\cup p''\cup \{\Seq{\Seq{ i,\alpha},1-\epsilon}\}$. Then $q\leq p'$ but also $q\Vdash \alpha\notin \dot{X}^{\dot{\sigma}}_i$ and $i\in \dom{\dot{\sigma}}$, contradiction. 
\end{proof}

The case $\mu = \omega$, $\kappa = \omega_1$ and $\lambda = 2^{\aleph_1}$ is the one of primary interest---this is the 
poset to add $2^{\aleph_1}$ Cohen reals (as a finite support product).
The proof can be summarized as follows, making clear that the conclusion carries over to the standard poset to add $2^{\aleph_1}$
random reals.
Let $\{X_i \mid i \in \lambda\}$ be the generic subsets of $\omega_1$.
First, if $I \subseteq \lambda$ is countably infinite and in the ground model, then both $\bigcap_{i \in I} X_i$ and 
$\bigcap_{i \in I} \omega_1 \setminus X_i$ are empty by genericity.
Second, if $I$ is a countably infinite set in the generic extension, then there is an $\alpha < \omega_1$ such that
$I$ is in the intermediate generic extension by $\{X_i \cap \alpha \mid i \in \lambda\}$.
Since $\{X_i \setminus \alpha \mid i \in \lambda\}$ is generic over this intermediate extension,
by the previous observation $\bigcap_{i \in I} X_i$ and $\bigcap_{i \in I} \omega_1 \setminus X_i$ 
are contained in $\alpha$.
Since any flip will be constant on a countably infinite set, the desired conclusion follows.

\begin{corollary}\label{ref: cohen extension and Tukey-top}
     In any generic extension by $\Add(\omega,2^{\omega_1})$ or by the homogeneous measure algebra of character
     $2^{\aleph_1}$,
     every uniform ultrafilter over $\omega_1$ is Tukey-top.
\end{corollary}

\begin{remark}
    This result was obtained independently by Jorge Chapital~\cite{JorgeBarely}.
\end{remark}

The results of~\cite{AbrahamShelah1986} also yield the following

\begin{corollary}
    In the Abraham-Shelah model from~\cite{AbrahamShelah1986}, the club filter over $\omega_1$ is $\omega_1$-Tukey-top, and moreover, every ultrafilter over $\omega_1$ is Tukey-top.
\end{corollary}
\begin{proof}
  Assume GCH and let $\mathbb{S}$ denote the Abraham-Shelah poset.
  Abraham and Shelah in~\cite{AbrahamShelah1986} prove that $\Add(\omega,\omega_2)$ is a regular suborder of $\mathbb{S}$ and that the quotient is $\sigma$-distributive, which is to say that the  any $\omega$-sequence of ordinals in $V^{\mathbb{S}}$ belongs already to $V^{\Add(\omega,\omega_2)}$.
  Hence the mutually-generic sequence of Cohen reals from the previous proof persists as a witness for any uniform ultrafilter over $\omega_1$ being Tukey-top.
\end{proof}

 Recall that in the Cancino and Zapletal model~\cite{CancinoZaplatal} where every ultrafilter over $\omega$ is Tukey-top, $2^{\aleph_0}=2^{\aleph_1}=\aleph_2$. 
\begin{theorem}
    Suppose that every ultrafilter over $\omega$ is Tukey-top and that $2^{\aleph_0}=2^{\aleph_1}$. Then every uniform ultrafilter over $\omega_1$ is Tukey-top. 
\end{theorem}
\begin{proof}
    Let $\Ucal$ be any uniform ultrafilter over $\omega_1$. It is well known that every uniform ultrafilter is $\omega$-decomposable, namely, there is a function $f\mathrel{:}\omega_1\to\omega$ such that $W=f_*(U)$ is a uniform ultrafilter over $\omega$. Since $\Wcal \leq_{RK} \Ucal$, we also have $\Wcal \leq_T \Ucal$, and since every ultrafilter over $\omega$ is Tukey-top we have:
    $$[2^{\aleph_1}]^{<\omega}=[2^{\aleph_0}]^{<\omega}\leq_T \Wcal \leq_T \Ucal.$$
    Hence $\Ucal$ is Tukey-top.
\end{proof}

\begin{remark}
    It was proposed (see for example~\cite{HrusakBook}) that a possible solution to the Katowice problem~\cite{Nykos} (whether it is consistent that $\mathscr{P}(\omega_1)/\fin$ and $\mathscr{P}(\omega)/\fin$ can consistently be isomorphic) is to prove in ZFC that there is an ultrafilter over $\omega$ that is not Tukey-equivalent to any uniform ultrafilter over $\omega_1$.
    This would give a negative answer to the problem.
    In the above model, this strategy cannot succeed.
    Still, $\mathscr{P}(\omega)/\fin \cong \mathscr{P}(\omega_1)/\fin$ has a number of nontrivial consequences and it is still conceivable that
    a combination of these consequences implies
    that there is an ultrafilter over $\omega$ which is not
    Tukey-equivalent to any uniform ultrafilter over $\omega_1$.
\end{remark}

Next, we consider whether CH implies the existence of a non-Tukey-top ultrafilter over $\omega_1$. First we observe that something slightly weaker than a flipping family with the bounded intersection property is needed:
\begin{proposition}
    Suppose that there is a family $\Seq{ A_\alpha\mid\alpha<\lambda}\subseteq \kappa$ such that whenever $I\in [\lambda]^\mu$, both $\bigcap_{\alpha\in I}A_\alpha$ and $\bigcap_{\alpha\in I}\omega_1\setminus A_\alpha$ are bounded in $\kappa$. Then every uniform ultrafilter over $\kappa$ is $(\mu,\lambda)$-Tukey-top.
\end{proposition}
\begin{proof}
    Given any uniform ultrafilter $\Ucal$ over a cardinal $\kappa$, for each $\alpha<\lambda$, either $A_\alpha\in \Ucal$ or $\kappa\setminus A_\alpha\in \Ucal$. There is $J\in [\lambda]^{\lambda}$ such that either $\{A_\alpha\mid \alpha\in J\}\subseteq \Ucal$ or $\{\kappa\setminus A_\alpha\mid \alpha\in J\}\subseteq \Ucal$. In either case, the assumption implies that $\Ucal$ is $(\mu,\lambda)$-Tukey-top.
\end{proof}
\begin{remark}
    The existence of $\Seq{ A_\alpha\mid\alpha<\lambda}\subseteq \kappa$ such that whenever $I\in [\lambda]^\mu$  both $\bigcap_{\alpha\in I}A_\alpha$ and $\bigcap_{\alpha\in I}\kappa\setminus A_\alpha$ are bounded in $\kappa$, is equivalent to the negative partition relation $\binom{\lambda}{\kappa}\not\to\binom{\mu}{\kappa}^{1,1}_2$; there exists $c\mathrel{:}\lambda\times\kappa\to 2$ such that there is no homogeneous rectangle (i.e. a set of the form $A\times B$ on which $c$ takes just one value), where $A\in [\lambda]^\mu$ and $B\in [\kappa]^{\kappa}$.
\end{remark}
P{\v r}ikr\'y~\cite{PRIKRYPolorized} showed the consistency of GCH with $\binom{\omega_2}{\omega_1}\not\to\binom{\omega}{\omega_1}^{1,1}_2$. Hence, we establish the following:
\begin{corollary}
    It is consistent that GCH holds and that every uniform ultrafilter over $\omega_1$ is Tukey-top. 
\end{corollary}
\begin{corollary}
    If $\binom{\omega_2}{\omega_1}\not\to\binom{\omega}{\omega_1}^{1,1}_2$ holds, then any cardinal-preserving $\sigma$-closed forcing will preserve this.
\end{corollary}
\begin{proof}
    Suppose $\mathbb{P}$ is $\sigma$-closed and $c\in V$ witnesses the negative partition relation. Suppose that $A\in V[G]$ is any countably infinite subset of $\omega_2$. Then $A\in V$. If $B\in V[G]$ is any subset of $\omega_1$ such that $c[A\times B]=\{i\}$, then $B\subseteq C:=\{\alpha<\omega_1\mid \forall a\in A, \ c(a,\alpha)=i\}$. Clearly, $C\in V$, and since $c[A\times C]=\{i\}$, $C$ must also be countable. Hence, $B$ is countable.
    \end{proof}
    
\section{Large Cardinal Ideals and Non-Tukey-Top Ultrafilters}\label{Section: Non-Tukeytop}
If $\kappa$ carries a uniform $\sigma$-complete ultrafilter (e.g. if $\kappa$ is measurable) then clearly this ultrafilter is non-Tukey-top.
Moreover, any $p$-point ultrafilter (in which case either $\kappa=\omega$ or else $\kappa$ has to be measurable)  will be a non-$\kappa$-Tukey-top ultrafilter.
Since, for example, there are no $p$-points nor $\sigma$-complete ultrafilters over $\omega_1$, it is unclear whether there can be a non-$\omega_1$-Tukey-top, or even a non Tukey-top, ultrafilter over $\omega_1$. In this subsection, we prove that such ultrafilters consistently exist over $\omega_1$. 
\begin{proposition}
    Suppose that $I$ is a $\sigma$-complete $(\kappa^+,\mu)$-saturated ideal over $\kappa$. Then the forcing $\mathscr{P}(\kappa)/I$ adds a $V$-ultrafilter which is not $(\mu,\kappa^+)$-Tukey-top.
\end{proposition}
\begin{proof}
    Let $G$ be the generic ultrafilter. Suppose toward a contradiction that $G$ is $\mu$-Tukey-top. Let $\Seq{ X_i\mid i<\kappa^+}$ witness that $G$ is $\mu$-Tukey-top. Let $(\dot{X}_i\mid i<\kappa^+)$ be names and let $Y\in \mathscr{P}(\kappa)/I$ be a condition such that $Y\Vdash \Seq{\dot{X}_i\mid i<\kappa^+}$ is a witness. For each $i<\kappa^+$ let $Y_i\leq Y$ be such that for some $Y_i\subseteq Z_i$, $Y_i\Vdash \dot{X}_i=\check{Z}_i$. Consider in $V$ the sequence $\Seq{Y_i\mid i<\kappa^+}$; by the saturation assumption there is a $\Xi \in [\kappa^+]^{\mu}$ such that $\Seq{Y_i\mid i\in \Xi}$ has a lower bound in $\mathscr{P}(\kappa)/I$. By the $\sigma$-completeness of $I$, $Y^*=\bigcap_{i\in \Xi}Y_i\in \mathscr{P}(\kappa)/I$.  However, $Y^*$ forces that $\Seq{\dot{X}_i\mid i<\kappa^+}$ does not witness that $G$ is Tukey-top, contradiction.
\end{proof}
Of course, the ultrafilter from the previous proposition is not going to be an ultrafilter in the generic extension. It is therefore natural to ask whether it is possible to construct a non-Tukey-top ultrafilter from an $(\omega_2,\omega_2,\omega)$-saturated ideal over $\omega_1$ or from other saturation assumptions. 
\begin{corollary}
    It is consistent that there is an $\omega_2$-saturated ideal and every uniform ultrafilter over $\omega_1$ is Tukey-top.
\end{corollary}
\begin{proof}
    Laver~\cite{LAVER1982297} showed that starting with a model where there is such an ideal $I$ and $CH$ holds, upon adding $\omega_2$-many Cohen reals, the filter generated by $I$ has the same saturation property. As we have seen in~\ref{ref: cohen extension and Tukey-top}, in this model every uniform ultrafilter over $\omega_1$ is Tukey-top.
\end{proof}

Let $(*)$ denote the assumption:
\[ \diamondsuit+\exists\text{ a normal ideal over }\omega_1 \text{ which is }\omega_1\text{-dense} \]
Woodin proved that $(*)$ is consistent relative to determinacy assumptions~\cite{woodinBook}. 
\begin{theorem}
    Under $(*)$ there is a weakly normal non-$\omega_1$-Tukey-top uniform ultrafilter over $\omega_1$.
\end{theorem}
\begin{proof} 
By Laver~\cite{LAVER1982297} there is an ultrafilter $U$ which is generated by  $I\cup\{A_\alpha\mid \alpha<\omega_1\}$, where $I$ is a normal filter and each $A_\alpha \subseteq \omega_1$. That is, for any $X\in U$, there is $\alpha<\omega_1$ such that $A_\alpha\setminus X\in I$. We claim that $U$ is not $\omega_1$-Tukey-top (and therefore also not Tukey-top). Let $\Seq{X_\alpha\mid \alpha<\omega_2}\subseteq U$. Then for every $\alpha<\omega_2$ there is $\beta<\omega_1$ such that $A_\beta\setminus X_\alpha\in I$. Fix $\beta^*$ and $J\in [\omega_2]^{\omega_2}$ such that for every $\alpha\in J$, $B_\alpha:=A_{\beta^*}\setminus X_\alpha\in I$. The sequence $\Seq{B_\alpha\mid \alpha\in J}$ is a sequence of $\omega_2$-many sets in the normal ideal $I$. Note that by $\diamondsuit$, CH holds and therefore we can apply Galvin's theorem~\ref{Thm: Galvin} and obtain $\omega_1$-many of the $B_\alpha$'s for which the union is in $I$. Choose $J_0\in[ J]^{\omega_1}$ such that $A_{\beta^*}\setminus(\bigcap_{\alpha\in J_0}X_\alpha)=\bigcup_{\alpha\in J_0} B_\alpha\in I$. Since $A_{\beta^*}\in U$, $\bigcap_{\alpha\in J_0}X_\alpha\in U$ as desired.  
\end{proof}
Huberich~\cite{Huberich} removed the diamond assumption and constructed a similar weakly normal ultrafilter from CH and an $\omega_1$-dense ideal over $\omega_1$. More precisely,  Huberich showed in~\cite[Corollary 11]{Huberich} that from a normal $\nu^+$-dense ideal $I$ over $\nu^+$ for $\nu$ regular, there is an ultrafilter $U \supseteq I^*$ over $\nu^+$ which is generated by $I^*\cup \{X_\alpha\mid \alpha<2^{\aleph_0}\}$. Let us use it to deduce that there is a non-Tukey-top ultrafilter from the weakening of $(*)$ in which $\diamondsuit$ is replaced by the weak diamond principle of Devlin and Shelah~\cite{DevlinShelah}, which is equivalent to $2^{\aleph_0}<2^{\aleph_1}$.
\begin{theorem}
    Suppose that there is a normal $\omega_1$-dense ideal over $\omega_1$ and that $2^{\aleph_0}<2^{\aleph_1}$. Then there is a non-($\omega_1$,$2^{\omega_1}$)-Tukey-top ultrafilter over $\omega_1$.
\end{theorem}
\begin{proof}
    By Huberich, let $U$ be an ultrafilter generated by $I\cup \{X_\alpha\mid \alpha<2^{\aleph_0}\}$. Given $2^{\aleph_1}$-many sets $\Seq{ A_\beta \mid \beta<2^{\aleph_1} }$, there is $J\in[2^{\aleph_1}]^{(2^{\aleph_0})^+}$ such that for some $\alpha^*<2^{\aleph_0}$, $X_{\alpha^*}\setminus A_\beta\in I$ for each $\beta\in J$. By Garti's generalization of Galvin's theorem~\cite[Thm 1.1]{Garti2017WeakDA}, there is $J'\in[J]^{\omega_1}$ such that $\bigcup_{j\in J'}(X_{\alpha^*}\setminus A_j)\in I$. We conclude that $\bigcap_{j\in J'}A_j\in U$.  
\end{proof}
Note that by P{\v r}ikr\'y and Jech~\cite[Thm. 7.2.1(a)]{PrikryJech}, if $\omega_1$ carries a non-regular ultrafilter and $2^{\aleph_0}<2^{\aleph_1}$, then necessarily  $2^{\aleph_0}\geq \aleph_{\omega_1}$.

\subsection{Non-regular and indecomposable ultrafilters}\label{Seciton: non-regular}
Since non-$(\mu,\lambda)$-regularity is a stronger form of non-$(\mu,\lambda)$-Tukey-top, it is tempting to ask whether other non-regular ultrafilters over $\kappa$, specifically non-$(\omega,\kappa)$-regular, can ever be Tukey-top.
 A related notion to that of non-regularity is the notion of indecomposability. 
 \begin{definition}
        An ultrafilter $U$ over $\kappa$ is \textit{$\nu$-decomposable} if there is a function $f\mathrel{:}\kappa\to \nu$ such that for every $X\in [\nu]^{<\nu}$, $f^{-1}(X)\notin U$. If there is no such function, we say that $U$ is \textit{$\nu$-indecomposible}
    \end{definition}
   Clearly, $U$ being $\nu$-decomposable is equivalent to $U$ being RK-above a uniform ultrafilter over $\nu$. 
    \begin{fact}
    If $U$ is $\nu$-indecomposable, then $U$ is not $(\omega,\nu)$-regular.
    \end{fact}

    Clearly, if $U$ is $\omega$-indecomposable, the $U$ is $\sigma$-complete and therefore non-Tukey-top. However, for $\nu>\omega$ the answer in general is negative:
\begin{proposition}
    Assume $CH$. Let $U$ be a uniform $\omega_1$-indecomposable ultrafilter over any cardinal $\kappa$ (even singular). Then after forcing with $\Add(\omega, 2^\kappa)$, $U$ can be extended to a Tukey-top $\omega_1$-indecomposable ultrafilter over $\kappa$.
\end{proposition}
\begin{proof}
    Note that in this case $\cf(\kappa)>\omega_1$, since any singular cardinal and any uniform ultrafilter over it must be $\cf(\kappa)$-decomposable. Theorem~\ref{Thm:TukeyTop in Cohen} applies to show that in the extension, every uniform ultrafilter over $\kappa$ is Tukey-top. We claim that $U$ generates a uniform $\omega_1$-indecomposable filter in $V[G]$. This is enough since any extension of this filter to a uniform ultrafilter will remain $\omega_1$-indecomposable and has to be Tukey-top. Indeed, let $\dot{f}$ be a name and $p$ a condition forcing $\dot{f}\mathrel{:}\kappa\to\omega_1$. We will prove that there is a set $X\in U$ such that $p$ forces $\dot{f}\restriction \check{X}$ is bounded. By the c.c.c. we can find in $V$, a function $F\mathrel{:}\kappa\to [\omega_1]^\omega$ such that $p$ forces that $\dot{f}(\alpha)\in \check{F}(\alpha)$ for every $\alpha<\kappa$. By $CH$ in the ground model, $F$ is essentially a function to $\omega_1$, so by $\omega_1$-indecomposability, there a set $X\in U$ such that $\bigcup F[X]$ is bounded in $\omega_1$. Hence $p$ forces that $\dot{f}\restriction\check{X}$ is bounded. 
\end{proof}
Another form of non-regularity is weak normality. As we have seen in the previous section, it is possible that a weakly normal ultrafilter is non-Tukey-top. 
It is natural to wonder if being non-Tukey-top just a consequence of weak normality.  This seems plausible in light of Galvin's theorem~\ref{Thm: Galvin}.
We will now show that this is not the case if $\kappa>\omega_1$ (see  question~\ref{Question: weakly normal omega1}).
\begin{theorem}
    Suppose that $U$ is a weakly normal ultrafilter over a regular $\kappa>\omega_1$, then after forcing with $\Add(\omega, 2^\kappa)$, $U$ generates a weakly normal filter which can be extended to a weakly normal ultrafilter which is Tukey-top.
 \end{theorem}
 \begin{proof}
     Let $\dot{f}\mathrel{:}\kappa\to \kappa$ be regressive in $V[G]$. Again, let $F\mathrel{:}\kappa\to [\kappa]^\omega$ cover $f$, and we may assume that $F(\alpha)\subseteq\alpha$ (as $f$ is regressive). Since $U$ is weakly normal, $X_0=\{\alpha\mid \cf(\alpha)>\omega\}\in U$. Hence $F(\alpha)$ is bounded in $\alpha$. By weak normality of $U$, there are $\beta<\kappa$ and $X\subseteq X_0$ in $U$ such that for every $\alpha\in X$, $F(\alpha)\subseteq \beta$. Then $\beta$ bounds $\dot{f}\restriction \check{X}$. Every extension of a weakly normal filter is a weakly normal ultrafilter by~\cite[Prop. 1.2]{KanamoriWeakNormal}. 
 \end{proof}

\subsection{A remark following Usuba}\label{Section: Usuba}
As we have seen, every ultrafilter over $\omega_1$ is $\omega$-decomposable and therefore every ultrafilter over $\omega_1$ is RK-above an ultrafilter over $\omega$. This raises the question of whether or not there can be two cardinals $\lambda<\kappa$ and a uniform ultrafilter $U_\kappa$ over $\kappa$, which is not Tukey-above any uniform ultrafilter over $\lambda$.
Recently, Usuba~\cite{Usuba} raised a similar question regarding the ultrafilter number and used both new and existing results regarding indecomposable ultrafilters to investigate the failure of monotonicity of the ultrafilter number function. The common theme, which we are next going to exploit in order to translate Usuba's results to the terminology of our investigation of the Tukey order, is the following:
\begin{proposition}
    Suppose $\lambda<\kappa$ and there is a uniform ultrafilter $U_\kappa$ over $\kappa$ such that  for every uniform ultrafilter $U_\lambda$ over $\lambda$, $U_\lambda\not\leq_T U_\kappa$. Then $U_\kappa$ is $\lambda$-indecomposable.
\end{proposition}
\begin{proof}
    If it is $\lambda$-decomposable then it RK-projects (and therefore Tukey reduces) to a uniform ultrafilter over $\lambda$.
\end{proof}
Let us denote by $TU(\lambda,\kappa)$ the statement that every uniform ultrafilter over $\kappa$ is Tukey-above a uniform ultrafilter over $\lambda$. The above proposition is saying that $TU(\lambda,\kappa)$ implies that there are no $\lambda$-indecomposable uniform ultrafilters over $\kappa$.

There are ZFC restrictions on the existence of indecomposable ultrafilters. These will be used in the following corollary:
\begin{corollary}
    \begin{enumerate}
        \item\label{TU-no-sigma-complete-equiv} For any cardinal $\kappa$, $TU(\omega,\kappa)$ if and only if $\kappa$ does not carry a uniform $\sigma$-complete ultrafilter.
        \item For any cardinal $\kappa$, $TU(\cf(\kappa),\kappa)$.
        \item For any regular cardinal $\kappa$, $TU(\kappa,\kappa^+)$.
        \item For any singular cardinal $\kappa$ of cofinality $\omega$ such that $\kappa^+$ does not carry a uniform $\sigma$-complete ultrafilter, $TU(\kappa,\kappa^+)$ holds.
        \item If $TU(\kappa,\kappa^+)$ fails, then there is a tail of regular cardinals $\mu<\kappa$ such that $TU(\mu,\kappa^+)$ holds for each $\mu$ in this tail.
    \end{enumerate}
\end{corollary}
\begin{proof}
    Only~(\ref{TU-no-sigma-complete-equiv}) does not directly follow from Usuba's paper~\cite{Usuba}. If $U$ is not $\sigma$-complete, then it is Tukey-above an ultrafilter over $\omega$. If it is $\sigma$-complete, then it cannot be above any ultrafilter over $\omega$, as follows from an easy argument using unbounded functions.
\end{proof}
 In~\cite{Usuba}, the author uses Raghavan and Shelah's~\cite{RagShel} to get the failure of monotonicity at many pairs of cardinals. Note that if $\mathfrak{u}(\kappa)<\mathfrak{u}(\lambda)$ then $\neg TU(\lambda,\kappa)$. In particular, we have the following consistency results which follow directly from~\cite{Usuba}:
 \begin{corollary}
     \begin{enumerate}
         \item Starting from a measurable cardinal $\kappa$, forcing with $\Add(\omega,\kappa^{+\omega_1})$ yields a model of $\neg TU(\kappa,\omega_1)$.
         \item Starting from a supercompact cardinal, it is consistent that after forcing with $\Add(\omega,\aleph_{\omega_1})$, $\neg TU(\omega_1,\omega_{\omega+1})$.
         \item After $\Pri(U)\times \Add(\omega,\kappa^{+\omega_1})$, $\kappa$ is singular of cofinality $\omega$ and $\neg TU(\kappa,\omega_1)$.
     \end{enumerate}
 \end{corollary}
The list above is by no means complete.  There are many other results that could be derived from known ones and questions that could be asked, but as our focus is mostly on ultrafilters over $\omega_1$, we leave this line of research unattended.

\section{The Ultrafilter \texorpdfstring{$\Ucal(T)$}{U(T)}}\label{Seciton: PFA}

In~\cite{l-maps}, Todorcevic defined a filter $\Ucal(T)$ associated to a
coherent A-tree $T$, and showed that if the countable chain condition is productive, then $\Ucal(T)$ is an ultrafilter.
He also established a number of additional properties of $\Ucal(T)$ under $\MA_{\omega_1}$ and $\PFA(\omega_1)$.

\begin{theorem}[{\cite{l-maps}}]
Assume $\PFA(\omega_1)$.
$\Ucal(T)$ is not RK-isomorphic
to an ultrafilter over $\omega_1$ which extends the club filter.
\end{theorem}

\begin{theorem}[{\cite{l-maps}}]
Assume $\PFA(\omega_1)$.
If $S$ and $T$ are two coherent A-trees, then $\Ucal(S)$ and $\Ucal(T)$ are
RK-isomorphic.
\end{theorem}

\begin{theorem}[{\cite{UT_selective}}]\label{UT_selective}
    Assume $\MA_{\omega_1}$.
    If $T$ is a coherent A-tree and $f\mathrel{:}\omega_1 \to \omega$ is not constant on
    any set in $\Ucal(T)$, then $f_*\Ucal(T)$ is a selective ultrafilter.
\end{theorem}

In this section we will add to this analysis, proving the following results.

\begin{theorem} \label{UT_Tukey-top}
Assume $\PFA(\omega_1)$.
For any coherent A-tree $T$, $[\omega_2]^{<\omega} \leq_T \Ucal(T)$.
In particular, PFA implies $\Ucal(T)$ is Tukey-top.
\end{theorem}

\begin{theorem} \label{UT_RK-min}
Assume $\PFA(\omega_1)$.
If $T$ is a coherent A-tree and $f\mathrel{:}\omega_1 \to \omega_1$, then
there is a set $U \in \Ucal(T)$ such that either $f$ is one-to-one on $U$ or 
$f$ is bounded on $U$.
\end{theorem}

In other words, $\PFA(\omega_1)$ implies that for any coherent A-tree
$T$, $\Ucal(T)$ is $\leq_{\RK}$-minimal with respect to being a uniform ultrafilter over $\omega_1$.
It was previously known that the cardinal arithmetic assumption $2^{\aleph_1} = \aleph_2$
(which follows from PFA)
already yields many RK-minimal uniform ultrafilters over $\omega_1$ \cite[Thm. 9.13]{ComfortNegro};
the point is that
$\Ucal(T)$ is a canonical example under suitable assumptions. 

It is known that a uniform ultrafilter over $\omega_1$ is weakly normal
if and only if it extends the club filter and
is $\leq_{\RK}$-minimal with respect to being uniform. 
Curiously enough, while Laver has shown that $\MA_{\omega_1}$ implies there are no weakly
normal ultrafilters over $\omega_1$ (see theorem~\ref{Thm: MA implies every ultrafilter is regular}), it consistent with $\MA_{\omega_1}$ that
$\Ucal(T)$ has either of these properties.

\begin{theorem} \label{UT_extends_club}
It is consistent with $\MA_{\omega_1}$ that there is a coherent A-tree $T$ such
that $\Ucal(T)$ extends the club filter.
\end{theorem}

Furthermore, since theorem~\ref{UT_selective} readily implies that any two projections
of $\Ucal(T)$ to $\omega$ are RK-isomorphic, theorem~\ref{UT_RK-min} yields the following corollary.

\begin{corollary} \label{UT_finest}
Assume $\PFA(\omega_1)$.
If $T$ is a coherent A-tree,
$\Ucal(T)$ has a finest partition.
\end{corollary} 
The consistent existence of ultrafilters over $\omega_1$ with a finest partition was previously demonstrated by Kanamori~\cite[p. 329]{finest_part}
using $\diamondsuit$ and the existence of an $\omega_1$-dense ideal over $\omega_1$. 

We now recall the definitions associated to $\Ucal(T)$.
A \emph{coherent A-tree} is a subset $T$ of
$\omega^{< \omega_1}$ such that:
\begin{itemize}
    \item $T$ is uncountable and closed under initial segments;
    \item if $s,t \in T$ have the same height, then $s =^* t$;
    \item there is no $b\mathrel{:}\omega_1 \to \omega$ such that for all $\alpha < \omega_1$, $b \restriction \alpha \in T$.
    
\end{itemize}
Recall that an A-tree $T$ is \emph{special} if there is a function $\varsigma \mathrel{:} T \to \omega$ such that
$\varsigma^{-1}(n)$ is an antichain for each $n$.
$\MA_{\omega_1}$ (and hence $\PFA(\omega_1)$) implies that all A-trees are special \cite{A-trees_special}.
Given a coherent A-tree $T$,
define
$$\Ucal(T) :=\{U \subseteq \omega_1 \mid \exists A \in [T]^{\omega_1} \ \Delta(A) \subseteq U\}.$$
Here $$\Delta(s,t):=\min \{\xi \in \dom(s) \cap \dom(t) : s(\xi) \ne t(\xi)\}$$
and $\Delta(A)$ is the set of all
$\Delta(s,t)$ such that $s,t \in A$ are incomparable. 
Most of the literature around $\Ucal(T)$ for coherent
A-trees assumes that $T$ has no Souslin subtrees---a
condition which is equivalent to them being \emph{Lipschitz} (see \cite[1.10]{l-maps});
clearly special A-trees have no Souslin subtrees.

\subsection{The Tukey-type of \texorpdfstring{$\Ucal(T)$}{U(T)}}\label{Subsec: U(T)Tukeytop}

We now prove theorem~\ref{UT_Tukey-top}.
Recall the definition of $U^{\vec{e}}_f$ from section~\ref{Section: omegaomega1}, where $f$ is a partial function from 
$\omega_1 \to \omega$:
\[
U^{\vec{e}}_f := \bigcap_{\alpha \in \dom(f)} U^{\vec{e}}_{\alpha,f(\alpha)}
\]
where 
\[
U^{\vec{e}}_{\alpha,n} := \{\beta \in \omega_1 \mid (\beta \leq \alpha) \lor (e_\beta(\alpha) \geq n) \}
\]
and $\Seq{e_\beta \mid \beta \in \omega_1}$ is any sequence such that each
$e_\beta\mathrel{:}\beta \to \omega$ is an injection.
In what follows $\vec{e}$ will be fixed and we will suppress it as a superscript for ease of reading.

\begin{lemma}
Assume $\PFA(\omega_1)$.
For any coherent A-tree $T \subseteq \omega^{<\omega_1}$ and any $f\mathrel{:}\omega_1 \to \omega$, there is 
a club $C \subseteq \omega_1$ such that $U_{f \restriction C}$ is in $\Ucal(T)$.
\end{lemma}

\begin{proof}
Let $T$ and $f$ be given and recall that $T$ must be special under our assumption.
Let $\varsigma:T \to \omega$ satisfy that $\varsigma^{-1}(n)$ is an antichain for each $n \in \omega$.
By proposition~\ref{Ulam_UF}, there is a least ordinal $\alpha_0 < \omega_1$ such that for all $\alpha \geq \alpha_0$ and all $n \in \omega$,
$U_{\alpha,n} \in \Ucal(T)$.
Define $Q$ to consist of all tuples $q = (A_q,C_q,O_q)$ such that:
\begin{enumerate}
    \item $A_q \subseteq T$ is a finite antichain;
    \item $O_q \subseteq \omega_1$ is a countable clopen set containing $[0,\alpha_0)$;
    \item $C_q \subseteq \omega_1 \setminus O_q$ is finite;
    \item if $s \ne t$ are in $A_q$ and $\alpha \in C_q$ with $\alpha < \beta:=\Delta(s,t)$,
    then $f(\alpha) \leq e_\beta(\alpha)$.
\end{enumerate}
We order $Q$ by coordinatewise reverse inclusion.
A countable elementary submodel of $H((2^{\aleph_1})^+)$ which has $T$, $\{e_\beta \mid \beta \in \omega_1\}$ and
$f$ as elements will be referred to as a \emph{suitable model} for $Q$.

\begin{claim} \label{generic_crit}
    If $M$ is a suitable model for $Q$, $q \in Q$ and $M \cap \omega_1 \in C_q$, then $q$ is $(M,Q)$-generic.
\end{claim}

\begin{proof}
    Set $\delta := M \cap \omega_1$ and let $D \subseteq Q$ be dense open and in $M$ with $q \in D$.
    It suffices to find an $r \in M \cap D$ such that $r$ is compatible with $q$.
    Let $\bar \delta < \delta$ be sufficiently large that:
    \begin{itemize}
        \item $O_q \cap \delta \subseteq \bar \delta$;

        \item $\{\Ht(s) \mid s \in A_q\} \cap \delta \subseteq \bar \delta$;

        \item if $s,t \in A_q$ with $\Delta(s,t) < \delta$, then $s(\xi) = t(\xi)$ for 
        all $\bar \delta \leq \xi < \delta$ (in particular $\Delta(s,t) < \bar\delta$);
        \item $C_q \cap \delta \subseteq \bar \delta$;

        \item if $\beta \in \Delta(A_q) \setminus \delta$, 
        then for all $\xi \in (\bar\delta,\delta)$,
        $$e_\beta(\xi) \geq \max\{f(\nu) \mid \nu \in C_q \setminus \delta\}.$$
        
    \end{itemize}
    If $r \in D$, define $n_r:=|C_r \setminus \bar \delta|$ and
    let $\delta^r_i$ be the $i\Th$-least element of $C_r \setminus \bar \delta$ for $i < n_r$.
    Let $t^r_i$ $(i < m_r)$ enumerate
    $$\{s \restriction \delta^r_0 \mid s \in A_r  \textrm{ and } \Ht(s) \geq \delta_0^r\}$$
    in $\leq_{\mathrm{lex}}$-increasing order.
    Set $m:=m_q$ and $n:= n_q$ and let $X \subseteq D$ consist of all $r$ such that:
\begin{itemize}

    \item $m_r = m$ and both $\varsigma(s^r_i) = \varsigma(s^q_i)$ and $s^r_i \restriction \bar\delta = s^q_i \restriction \bar \delta$
    for all $i < m$;

    \item $n_r = n$ and $f(\delta^r_i) = f(\delta^q_i)$ for each $i < n_q$;
    
    \item $\{s \in A_r \mid \Ht(s) < \delta^r_0\} = \{s \in A_q \mid \Ht(s) < \delta^q_0\}$;

    \item $C_q \cap \delta^q_0 = C_r \cap \delta^r_0$.

\end{itemize}
We will eventually select an $r$ from $X \cap M$ which is compatible with $q$---i.e. such
that $\bar r:=(A_q \cup A_r,C_q \cup C_r,O_q \cup O_r)$ is a condition.
In fact, many of the requirements necessary to be a member of $Q$ are automatically met by $\bar r$ just
by virtue of $r$ being in $X \cap M$.
First observe that if $r \in X \cap M$, then 
$\max(O_r) < \delta = \delta^q_0$ and hence $C_r \cup C_q$ is disjoint from $O_r \cup O_q$.
It is also true that $r \in X \cap M$
implies that $A_{\bar r}$ is an antichain.
To see this, suppose that $t \in A_q \setminus A_r$ and $t' \in A_r \setminus A_q$.
Notice that it must be that $\Ht(t) \geq \delta^q_0$ and $\Ht(t') \geq \delta^r_0$.
Let $i,i' < m$ be such that $t \restriction \delta^q_0 = s^q_i$ and $t' \restriction \delta^r_0 = s^r_{i'}$.
If $i \ne i'$, then since
\[
s^r_i \restriction \bar \delta = s^q_i \restriction \bar \delta \ne s^q_{i'} \restriction \bar \delta = s^r_{i'} \restriction \bar \delta
\]
it must be that $s^r_i$ and $s^q_{i'}$ are incompatible.
If $i = i'$, then
since $\varsigma(s^r_i) = \varsigma(s^q_i)$ and $s^r_i \ne s^q_i$ (since for instance $s^r_i \in M$ and $\Ht(s^q_i) = \delta \not \in M$),
$s^r_i$ and $s^q_{i'}$ are incompatible.
Since $t$ extends $s^q_i$ and $t'$ extends $s^r_i$, $t$ and $t'$ are incompatible.

Next suppose that $\beta$ is in $\Delta(A_{\bar r}) = \Delta(A_q \cup A_r)$.
If $\beta \in \Delta(A_r)$, then $\beta < \delta = \delta^q_0$.
In particular if $\alpha \in C_{\bar r}$ with $\alpha < \beta$, $\alpha \in C_r$.
Thus $f(\alpha) \leq e_\beta(\alpha)$ by virtue of $r$ being a condition.
If $\beta \in \Delta(A_q)$ and $\alpha \in C_{\bar r} \setminus C_q$, let $i < n$ be such that $\alpha = \delta^r_i$.
Observe that $f(\alpha) = f(\delta^r_i) = f(\delta^q_i)$.
Since $\alpha \in (\bar \delta,\delta)$, it follows that $e_\beta(\alpha) \geq f(\alpha)$.

The remaining possibility is that $\beta \in \Delta(A_q \cup A_r) \setminus (\Delta(A_q) \cup \Delta(A_r))$---
that is $\beta = \Delta(s^q_i,s^r_i)$ for some $i < m$.
Observe that since for all $j < m$
\[
\Delta(s^q_j,s^r_j) \in (\bar \delta,\delta^r_0) \subseteq (\bar \delta,\delta)
\]
and since for all $j < m$
\[
s^q_j \restriction (\bar \delta,\delta) = s^q_0 \restriction (\bar \delta,\delta),
\]
\[
s^r_j \restriction (\bar \delta,\delta^r_0) = s^r_0 \restriction (\bar \delta,\delta^r_0),
\]
it follows that $\beta = \Delta(s^q_i,s^r_i) = \Delta(s^q_j,s^r_j)$ for all $i,j < m$. 
Thus we need to select an $r$ such that if $\beta = \Delta(s^q_0,s^r_0)$,
then for all $\alpha \in C_r \cap \bar \delta$, $f(\alpha) \leq e_\beta(\alpha)$.
Notice that since $\alpha_0 \subseteq O_r$, $\alpha_0 \leq \min C_r$, and therefore
$$U:=\bigcap \{U_{\alpha,{f(\alpha)}} \mid \alpha \in C_r\}$$ is in $\Ucal(T)$.
Set $Y:=\{s^r_0 \mid r \in X\}$, noting that $Y \subseteq T$ is an antichain and $s^q_0$ is in $Y$.
Furthermore $Y$ is in $M$ since it is definable from parameters in $M$.
Let $Z$ be the set of all $s \in Y$ such that for all $s' \ne s$ in $Y$, $\Delta(s,s') \not \in U$.
Since $Z$ is definable from parameters in $M$, $Z$ is in $M$.
If $s^q_0$ were in $Z$, then $Z$ would have uncountable $\Delta(Z)$ and $U$ would contain two disjoint sets in $\Ucal(T)$, which
contradicts that $\Ucal(T)$ is a filter.
Thus $s^q_0 \not \in Z$ and hence there is an $r \in Z$ such that $\Delta(s^q_0,s^r_0) \in U$.
By our above observations, $\bar r = (A_q \cup A_r,C_q \cup C_r,O_q \cup O_r)$ is a condition witnessing
that $q$ is compatible with $r \in D \cap M$ as desired.
\end{proof}

\begin{claim} \label{proper_density}
     $Q$ is proper and the following sets are dense below some $p_0 \in Q$ for each $\xi \in \omega_1$:
    \begin{itemize}
    
       \item $\{q \in Q \mid \xi < \max C_q\}$,

        \item $\{q \in Q \mid \xi < \max \{\Ht(s) \mid s \in A_q\}\}$.

    \end{itemize}
\end{claim}

\begin{proof}
    To see that $Q$ is proper, let $M$ be suitable for $Q$, and $p \in Q \cap M$.
    Define $q = (A_p,C_p \cup \{\delta\},O_p)$, where $\delta = M \cap \omega_1$.
    By claim~\ref{generic_crit}, $q$ is a $(M,Q)$-generic condition.
    Since $p$ and $M$ were arbitrary, $Q$ is proper.

    Next suppose that $M$ is suitable for $Q$ and let $t \in T \setminus M$.
    Define $p_0 :=(\{t\},\{\delta\},\emptyset)$, where $\delta = M \cap \omega_1$.
    By claim~\ref{generic_crit}, $p_0$ is $(M,Q)$-generic.
    It follows that $p_0$ forces that $M[\dot G \cap M]$ is elementary in $H((2^{\aleph_1})^+)[\dot G]$
    and that $\bigcup \{A_q \mid q \in \dot G\}$ and $\bigcup \{C_q \mid q \in \dot G\}$ are both not contained in
    $M[\dot G \cap M]$ and hence that both are uncountable.
    It follows that for any $\xi \in \omega_1$,
    $\{q \in Q \mid \xi < \max C_q\}$ and $\{q \in Q \mid \xi < \max \{\Ht(s) \mid s \in A_q\}\}$ are
    dense below $p_0$.
\end{proof}

\begin{claim} \label{density}
    For all $\xi \in \omega_1$, $D:=\{q \in Q \mid \xi \in C_q \cup O_q\}$ is dense.
\end{claim}

\begin{proof}
Toward this end, let $p \in Q$ be given and observe that if $\xi \in C_q$, then $p \in D$.
If $\xi \not \in C_q$, then there is a $\bar \xi < \xi$ such that $(\bar \xi,\xi] \cap C_q = \emptyset$.
In this case $q:=(A_p,C_p,O_p \cup (\bar \xi,\xi])$ is an extension of $p$ in $D$ as desired.
\end{proof}

Let $G \subseteq Q$ be a filter containing $p_0$ and meeting the dense sets listed in claims~\ref{proper_density} and \ref{density} for each $\xi < \omega_1$.
Define $A:=\bigcup \{A_q \mid q \in G\}$, $C:=\{C_q \mid q \in G\}$, and $O:=\{O_q \mid q \in G\}$.
Clearly $A \subseteq T$ is an uncountable antichain and $C \subseteq \omega_1$ is uncountable.
Since $O$ is open and is the complement of $C$, $C$ is club.
Finally, set
\[
U := \{\Delta(s,t) \mid s \ne t \textrm{ and } s,t \in A\}.
\]
By definition, $U \in \Ucal(T)$.
Moreover, if $\alpha \in C$ and $\beta = \Delta(s,t) \in U$, then there must be some $q \in G$ such that
$\alpha \in C_q$ and $s,t \in A_q$.
Thus $f(\alpha) \leq e_\beta(\alpha)$.
Consequently we have shown that $U \subseteq U_{f \restriction C}$ and hence that $U_{f \restriction C}$ is in $\Ucal(T)$.
\end{proof}

Recall that for an ultrafilter $\Vcal$ 
$(\dagger)_{NS_{\omega_1},\Vcal}$ is the following statement:
$$\forall f\in \omega^{\omega_1}\exists C_f\text{ club such that }U_{f \restriction C_f} \in \Vcal$$
The previous lemma asserts that under PFA, every ultrafilter of the form $\Ucal(T)$ satisfies $(\dagger)_{NS_{\omega_1},\Ucal(T)}$. Theorem~\ref{UT_Tukey-top} then follows immediately from corollary~\ref{the: (PFA) Dagger implies tukey top}.

\subsection{\texorpdfstring{$\Ucal(T)$}{U(T)} is RK-minimal}

We now turn to the proof of theorem~\ref{UT_RK-min}.
\begin{proof}
Let $T$ be a special coherent A-tree and $f\mathrel{:}\omega_1 \to \omega_1$ be given.
Recall that $\MA_{\omega_1}$
implies $\Ucal(T)$ is an ultrafilter and that $T$ is special.
Fix a function $\varsigma:T \to \omega$ such that $\varsigma^{-1}(n)$ is an antichain for all $n$.
If there is an $\alpha < \omega_1$ such that
$\{\delta \in \omega_1 \mid f(\delta) \leq \alpha\}$
is in $\Ucal(T)$, then we are finished.
Thus we may assume that for all $\alpha < \omega_1$
$$\{\delta \in \omega_1 \mid \alpha < f(\delta)\} \in \Ucal(T).$$
Define
$Q$ to be the set of all pairs $q = (E_q,A_q)$ such
that:
\begin{enumerate}

\item $E_q \subseteq \omega_1$ is finite;

\item $A_q \subseteq T$ is a finite antichain such that
$f \restriction \Delta(A_q)$ is one-to-one;

\item if $\nu \in E_q$ and $s\ne t \in A_q$ are such that
$\nu \leq \Delta(s,t)$, then $\nu \leq f(\Delta(s,t))$.

\end{enumerate}
\begin{claim} \label{RK_minimal_proper}
If $M$ is a countable elementary submodel of $H(\omega_2)$
with $T,f,\varsigma \in M$ and $q \in Q$ is such that
$M \cap \omega_1 \in E_q$, then $q$ is $(M,Q)$-generic.
In particular $Q$ is proper.
\end{claim}

\begin{proof}
Let $q$ be given and $D \subseteq Q$ be a dense set in $M$.
We need to find a $p \in D \cap M$ such that $p$ is compatible
with $q$.
By extending $q$ if necessary, we may assume that $q \in D$.
Set $\nu := M \cap \omega_1$ and let $\nu' < \nu$ be
sufficiently large that:
\begin{itemize}

\item if $\delta < \nu'$, $f(\delta) < \nu'$;

\item if $s \in A_q \cap M$, $\Ht(s) < \nu'$;

\item if $s, t \in A_q \setminus M$, then $s(\xi) = t(\xi)$
for all $\nu' \leq \xi < \nu$;

\end{itemize}
Since $f$ is not bounded on any set in $\Ucal(T)$ and since
$\Ucal(T)$ is an ultrafilter,
there is an uncountable antichain
$X \subseteq T$ such that
if $s,t \in X$ are distinct, then $f(\Delta(s,t)) > \nu'$.
Fix a function $p \mapsto t_p$ in $M$ with domain $D$
such that for all $p$, $t_p \in T$ has height $\min(E_p \setminus \nu')$ and $t_p$ is extended by an element of $X$.
Let $\nu'' < \nu$ be sufficiently large that $\nu' \leq \nu''$ and if
$\xi < \nu$ and $t_q(\xi) \ne s(\xi)$ for some $s \in A_q$,
then $\xi < \nu''$.

Let $D'$ consist of those elements $p$ of $D$ such that:
\begin{itemize}

\item $|E_p| = |E_q|$ and there is a (necessarily unique)
$\nu_p \in E_p$ such that $\nu_p \cap E_p = \nu_q \cap E_q$ and $\nu'' < \nu_p$,

\item $|A_p| = |A_q|$, $\{s \in A_p \mid \Ht(s) < \nu_p\} = \{s \in A_q \mid \Ht(s) < \nu_q\}$, and
$$\{s \restriction \nu'' \mid (s \in A_p) \land (\Ht(s) \geq \nu_p)\} = \{s \restriction \nu'' \mid (s \in A_q) \land (\Ht(s) \geq \nu_q)\};$$

\item if $s \in A_p$ and $\Ht(s) \geq \nu_p$, then whenever $\nu'' \leq \xi < \nu_p$, $s(\xi) = t_p(\xi)$;

\item $\varsigma(t_p) = \varsigma(t_q)$ and $t_p \restriction \nu'' = t_q \restriction \nu''$;

\end{itemize}
Observe that $\nu_q = \nu$,
$q \in D'$, and that $D'$ is definable from the parameters 
$E_q \cap \nu_q$,
$\{s \in A_q \mid \Ht(s) < \nu_q\}$,
$\{s \restriction \nu'' \mid (s \in A_q) \land (\Ht(s) \geq \nu_q)\}$, and
$t_q \restriction \nu''$ which are each in $M$.
Thus $D' \in M$.

We claim that any element $p$ of $D' \cap M$ is compatible
with $q$.
It suffices to show that $r:=(E_p \cup E_q,A_p \cup A_q)$ is a condition in $Q$.
First observe that since $\varsigma(t_p) = \varsigma(t_q)$ and $\Ht(t_p) < \nu =\Ht (t_q)$, $t_p$ is incompatible with $t_q$;
let $\delta = \Delta(t_p,t_q)$.
If $s \in A_p \setminus A_q$ and $s' \in A_q \setminus A_p$,
then by definition of $D'$ and the fact that $p \in D' \cap M$, we know that $\nu'' < \nu_p \leq \Ht(s) < \nu \leq \Ht(s')$.  
Since $t_p \restriction \nu'' = t_q \restriction \nu''$, it follows that $\nu'' \leq \delta < \nu_p < \nu$ and 
consequently $s(\delta) = t_p(\delta) \ne t_q(\delta) = s'(\delta)$.
In particular, $s$ and $s'$ are incompatible.
Furthermore, if $\Delta(s,s') \not \in \Delta(A_p)$,
then again by definition of $D'$, it must be that
$\Delta(s,s') \geq \nu''$.
Since $s$ agrees with $t_p$ on $[\nu'',\delta)$ and
$s'$ agrees with $t_q$ on $[\nu'',\delta)$, it follows
that $\Delta(s,s') = \Delta(t_p,t_q) = \delta$.
Summarizing, we have shown that $A_p \cup A_q$ is an 
antichain and $\Delta(A_p \cup A_q) = \Delta(A_p) \cup \Delta(A_q) \cup \{\delta\}$.
Notice that by this argument,
if $p \in D'$, 
$\Delta(A_p) \setminus \nu' = \Delta(A_p) \setminus \nu_p$.

In order to show that $r$ is a condition, it remains to show that $f$ is one-to-one when restricted to
$\Delta(A_r)$.
Observe that $\Delta(A_p) \cap \nu' = \Delta(A_q) \cap \nu'$ and that
$$
\Delta(A_p) \cap \nu' < 
\nu' \leq \delta < \nu_p \leq \Delta(A_p) \setminus \nu' < \nu \leq \Delta(A_q) \setminus \nu'. 
$$
Also, $\nu'$, $\nu_p$, and $\nu$ are closed under $f$.
Additionally, by virtue of $p$ being a condition in $Q$,
if $\delta' \in \Delta(A_p) \setminus \nu'$, $f(\delta') \geq \nu_p$.
Similarly if $\delta' \in \Delta(A_q) \setminus \nu'$,
$f(\delta') \geq \nu$.
It follows that $f$ is one-to-one when restricted to
$\Delta(A_p) \cup \Delta(A_q)$.
Finally, since $\delta \in \Delta(X)$ and $\nu_p$ is $f$-closed,
$\nu' \leq f(\delta) < \nu_p$.
It follows that $f$ is one-to-one on 
$\Delta(A_r) = \Delta(A_p) \cup \Delta(A_q) \cup \{\delta\}$ as
well.
\end{proof}

\begin{claim} \label{RK_minimal_density}
There is a condition $q \in Q$ which forces that 
$\dot A := \bigcup \{A_p \mid p \in \dot G\}$ is an uncountable
antichain such that $f \restriction \Delta(\dot A)$
is one-to-one.
\end{claim}

\begin{proof}
Since it is forced that $\dot A$ is a directed union of antichains, it is forced to be an antichain.
Similarly, it is forced that $\check f \restriction \Delta(\dot A)$ is one-to-one.
Let $M$ be a countable elementary submodel of a sufficiently
large $H_\theta$ such that $T,f,\varsigma \in M$
and let $t \in T \setminus M$.
Since $q = (\{M \cap \omega_1\},\emptyset,\{t\})$, 
$q$ is $(M,Q)$-generic by claim~\ref{RK_minimal_proper}.
Because $q$ forces $M[\dot G]$ is elementary in
$H(\theta)[G]$,
it follows that $q$ forces $\dot A$ and $\dot E$ are uncountable.
\end{proof}

To finish the proof of theorem~\ref{UT_RK-min}, let $q$ force that $\dot A$ is uncountable
and $D_\xi$ consist of those extensions $p$ of $q$ such that $A_p$ contains an element of height at least $\xi$.
By claim~\ref{RK_minimal_density}, each $D_\xi$ is dense below $q$.
By $\PFA(\omega_1)$,
there is a filter $G$ which intersects $D_\xi$ for each $\xi \in \omega_1$.
If $A = \bigcup_{p \in G} A_p$, then $\Delta(A)$ is in $\Ucal(T)$ and $f$ is one-to-one on $\Delta(A)$.
\end{proof}

\subsection{\texorpdfstring{$\Ucal(T)$}{U(T)} can extend the club filter}
Our next goal is to prove the following theorem from which theorem~\ref{UT_extends_club} follows.
We will often need to refer to $\Ucal(T)$ in generic extension for a given $T$.
In all cases, $\Ucal(T)$ will be interpreted in the generic extension and we add a ``dot''
to emphasize this.
Thus $\dot \Ucal(\check T)$ is the name for the filter $\Ucal(T)$ computed in the generic extension for a coherent tree $T$ from the ground model.

\begin{theorem} \label{th:ltree-c.c.c.-extend-club}
    There is a c.c.c. poset which forces
    $\MA_{\omega_1}$ and ``there is a coherent A-tree $T$ such that $\dot \Ucal(\check T)$ extends the club filter.''
\end{theorem}

\begin{lemma} \label{gentle_specialization}
    Suppose that $T$ is a coherent Souslin tree such that every element has at least two immediate successors.
    There is a c.c.c. poset which forces ``$\check T$ is special'' and,
    for all uncountable $X \subseteq \omega_1$, forces ``$\check X \in \dot \Ucal(\check T)^+$.'' 
\end{lemma}

\begin{remark}
    In particular,
    since every club subset of $\omega_1$ in a c.c.c. forcing extension contains a ground model club, 
    the poset in lemma~\ref{gentle_specialization} forces
    $\dot \Ucal(\check T) \cap NS_{\omega_1} = \emptyset$.
\end{remark}
\begin{proof}
    Let $T$ be given and let $\Abb$ be the finite-support countable power of the poset of all finite antichains
    of $T$---this is the standard c.c.c. poset to specialize $T$.
    Thus it suffices to show that if $X \subseteq \omega_1$ is uncountable,
    $\Abb$ forces $\check X \in \dot \Ucal(\check T)^+$.
    Toward this end, let $X \subseteq \omega_1$ be given and let $\dot A$ be an $\Abb$-name such that $p \in \Abb$
    forces $\dot A$ is uncountable subset of $\check T$.
    For each $\xi \in \omega_1$, let $p_\xi$ be an extension of $p$ and $t_\xi$ be an element of $T$ of height at least $\xi$ such 
    that $p_\xi \Vdash \check t_\xi \in \dot A$.
    For each limit ordinal $\xi$, let $r(\xi) < \xi$ be such that:
    \begin{itemize}
        \item if $i \in \dom(p_\xi)$ and $s \in p_\xi(i)$ has height less than $\xi$, it has height less than $r(\xi)$;
        \item if $i \in \dom(p_\xi)$ and $s \in p_\xi(i)$ has height at least $\xi$, then $s(\eta) = t_\xi(\eta)$ whenever
        $r(\xi) \leq \eta < \xi$.
    \end{itemize}
    By the pressing-down lemma there is a stationary set $\Xi \subseteq \omega_1$ such that $r$ is constantly $\zeta$ on $\Xi$.
    By further refining $\Xi$ if necessary, we may assume that $n:=\dom(p_\xi)$ does not depend on $\xi$ and for each $i < n$,
    the set of elements of $p_\xi(i)$ of height less than $\zeta$ does not depend on $\xi$.
    Observe that if $t_\xi \restriction \xi$ is incompatible with $t_{\xi'} \restriction \xi'$, then $p_\xi(i) \cup p_{\xi'}(i)$
    is an antichain for all $i < n$ and hence $p_\xi$ is compatible with $p_{\xi'}$.
    Since $T$ is Souslin, the downward closure of $\{t_\xi \mid \xi \in \Xi\}$ contains a cone in $T$.
    In particular there are $\xi \ne \xi'$ in $\Xi$ such that $t_\xi$ and $t_{\xi'}$ are incompatible and
    $\Delta(t_\xi,t_{\xi'})$ is in $X$.
    It follows that $p_\xi$ and $p_{\xi'}$ are compatible and any common extension forces $\Delta(t_\xi,t_{\xi'}) \in \Delta(\dot A) \cap \check X$.
\end{proof}

\begin{definition}
Let $T$ be a coherent special A-tree.
For $X \subseteq \omega_1$, $Q_{T,X}$ is the poset consisting of finite antichains $q \subseteq T$ such that $\Delta(q) \subseteq X$, ordered by reverse inclusion.
If $T$ is clear from context, we will write $Q_X$.
\end{definition}

\begin{lemma}[{\cite[rmk. 4.3]{l-maps}}] \label{lem:atree-c.c.c.-coideal}
Let $T$ be a coherent special A-tree.
For $X \subseteq \omega_1$, $Q_X$ is c.c.c. if and only if $X \in \Ucal(T)^+$.
\end{lemma}

The next lemma is essentially due to Todorcevic; see Lemmas 1.3 and 1.9 of \cite{l-maps} as well as their proofs.

\begin{lemma} \label{new_Delta}
    Suppose $T$ is a special coherent A-tree.
    If $\{q_\xi\mid \xi \in \omega_1\}$ is an uncountable family
    of finite subsets of $T$ and $t_\xi \in T$ has height at least $\xi$, then there is an uncountable
    $\Xi \subseteq \omega_1$ such that for all $\xi \ne \eta \in \Xi$,
    $$\Delta(q_\xi \cup q_\eta) = \Delta(q_\xi) \cup \Delta(q_\eta) \cup \{\Delta(t_\xi,t_\eta)\}.$$
\end{lemma}

\begin{lemma} \label{liptree-delta-stable}
Let $T$ be a coherent special A-tree.
For $X_0,\ldots,X_{n-1} \subseteq \omega_1$, $\prod_{i < n} Q_{X_i}$ is c.c.c. if and only if $\bigcap_{i < n} X_i \in \Ucal(T)^+$.
\end{lemma}

\begin{proof}
    Assume first that $\prod_{i < n} Q_{X_i}$ is c.c.c..
    It follows that there is a $q \in \prod_{i < n} Q_{X_i}$ which forces that
    $\dot A:=\{t \in T \mid (\{t\},\ldots,\{t\}) \in \dot G\}$ is uncountable.
    Since every condition forces $\Delta(\dot A) \subseteq \bigcap_{i < n} \check X_i$, it follows that $q$ forces
    $\bigcap_{i < n} \check X_i \in \dot \Ucal(\check T)$.
    Because membership in $\Ucal(T)$ is upwards absolute and since $\Ucal(T)$ is a filter in all generic extensions in which $T$
    remains uncountable, it follows that $\omega_1 \setminus \bigcap_{i < n} X_i \notin \Ucal(T)$ or, equivalently,
    that $\bigcap_{i < n} X_i \in \Ucal(T)^+$.
    
    Now consider the converse and assume $\bigcap_{i < n} X_i \in \Ucal(T)^+$.
    For each $\xi \in \omega_1$, fix $t_\xi \in T$ of height $\xi$.    
    Consider a collection $\{ (q^0_\xi, \ldots, q^{n-1}_\xi) \mathrel{:} \xi < \omega_1 \} \subseteq \prod_{i < n} Q_{X_i}$.
    By $n$ applications of lemma~\ref{new_Delta}, there is an uncountable $\Xi \subseteq \omega_1$ 
    such that for all $\xi \ne \eta$ in $\Xi$ and $i < n$,
    $$\Delta(q^i_\xi \cup q^i_\eta) \subseteq \Delta(q^i_\xi) \cup \Delta(q^i_\eta) \cup \{\Delta(t_\xi,t_\eta)\}.$$
    Since $\bigcap_{i < n} X_i \in \Ucal(T)^+$, there are $\xi \ne \eta$ in $\Xi$ such that
    $\Delta(t_\xi,t_\eta) \in \bigcap_{i < n} X_i$.
    Because $\Delta(q^i_\xi) \cup \Delta(q^i_\eta) \subseteq X_i$ by virtue of $q^i_\xi,q^i_\eta \in Q_{X_i}$,
    it follows that
    $$
    \Delta(q^i_\xi \cup q^i_\eta) \subseteq \Delta(q^i_\xi) \cup \Delta(q^i_\eta) \cup \{\Delta(t_\xi,t_\eta)\} \subseteq X_i.
    $$
    Thus $(q^0_\xi \cup q^0_\eta,\ldots, q^{n-1}_\xi \cup q^{n-1}_\eta)$ is a common extension of
    $(q^0_\xi,\ldots,q^{n-1}_\xi)$ and $(q^0_\eta,\ldots,q^{n-1}_\eta)$.
    Therefore $\prod_{i< n}Q_{X_i}$ is c.c.c..
\end{proof}

\begin{proposition} \label{Fcal->U(T)}
    Suppose that $T$ is a coherent Souslin tree such that every element has at least two immediate successors
    and $\Fcal$ is a filter on $\omega_1$ containing all cobounded sets.
    There is a c.c.c. poset $\Qbb$ which forces $\check \Fcal \subseteq \dot \Ucal(\check T)$.
\end{proposition}

\begin{proof}
Let $\Abb$ be the countable finite-support power of the poset of finite antichains of $T$.
Work in a forcing extension by $\Abb$, noting that in this extension $T$ is special and by lemma~\ref{gentle_specialization},
$\Fcal$ is a filter contained in $\Ucal(T)^+$.
Let $\Qbb$ be the finite support product of the posets $Q_X^{<\omega}$ for $X \in \Fcal$, where $Q_X^{<\omega}$ denotes the finite-support countable power of $Q_X$.
Since a finite-support product is c.c.c. if and only all of its finite subproducts are c.c.c., it suffices to show that if
$\Seq{ X_i \mid i < n }$ is a finite sequence of elements of $\Fcal$ (possibly with repetition),
then $\prod_{i < n} Q_{X_i}$ is c.c.c..
As  $\bigcap_{i < n} X_i \in \Fcal \subseteq \Ucal(T)^+$,
lemma~\ref{liptree-delta-stable} implies that $\prod_{i < n} Q_{X_i}$ is c.c.c..
Finally, since $Q_{X}^{<\omega}$ forces that $\widecheck{Q_X}$ is a union of countably many filters,
it forces that $\check X \in \dot \Ucal(\check T)$.
Thus $\Qbb$ forces $\check \Fcal \subseteq \dot \Ucal(\check T)$.
\end{proof}

\begin{proof}[Proof of theorem~\ref{th:ltree-c.c.c.-extend-club}]
By \cite[6.9]{Todorcevic1987}, in any generic extension by $\Add(\omega,1)$ there is a coherent Souslin tree $T$, which we may take
to have the property that every element has at least two immediate successors.
By proposition~\ref{Fcal->U(T)} applied to this Souslin tree and the club filter,
there is an $\Add(\omega,1)$-name $\dot \Pbb$ for a c.c.c. poset such that $\Add(\omega,1) * \dot \Pbb$
forces that $\Ucal(\dot T)$ extends the club filter.
Now let $\dot \Qbb$ be any $\Add(\omega,1)* \dot \Pbb$-name for a c.c.c. poset which forces $\MA_{\omega_1}$.
Since any club in a c.c.c. forcing extension contains a ground model club,
in the final extension $\Ucal(T)$ will still extend the club filter
and additionally $\MA_{\omega_1}$ will hold (in particular $\Ucal(T)$ will be an ultrafilter).
\end{proof}

\begin{corollary}
    It is consistent that $\MA_{\omega_1}$ holds and
    there is a coherent A-tree $T$ such
    that $\Ucal(T)$ extends the club filter and is $\omega_1$-Tukey-top.
\end{corollary}
\begin{proof}
Start with the Abraham-Shelah model and go into the
c.c.c. forcing extension in which there is a coherent A-tree $T$ such that $\Ucal(T)$ extends the club filter.
The desired conclusion follows from the properties of
the club filter in the Abraham-Shelah model.
\end{proof}

\section{Questions and Further Directions}
\label{Section: Questions}
In this section, we collect some problems and proposed directions. First, the results of this paper emphasize the connection between the Tukey types of ultrafilters over uncountable cardinals and the Tukey types of function spaces of the form $\mu^\lambda$. To the best of our knowledge, the study of such Tukey types is lacking. Particularly, we would like to know the answer to the following question (see StackExchange discussion~\cite{MathOverfloow}):
\begin{question}\label{Question: omegaomega1}
    Is it provable in ZFC that $\omega^{\omega_1}$ is Tukey-top?
\end{question}
\noindent
Note that it is an old problem of P{\v r}ikr\'y whether ZFC proves $\cf (\omega^{\omega_1},\leq) = 2^{\aleph_1}$.

In section~\ref{Section: Non-Tukeytop}, we showed that consistently every uniform ultrafilter over $\omega_1$ is Tukey-top, and provided several models for it. There are other models of interest where we do not know whether all uniform ultrafilters over $\omega_1$ are Tukey-top. It is plausible that the rigidity of the structure of ultrafilters over $\omega$ in some of these models would enable a proof of the independence of Isbell's question for uncountable cardinals over ZFC, with no need for large cardinals such as our present construction~\ref{Section: Non-Tukeytop} of non-$\omega_1$-Tukey-top ultrafilters over $\omega_1$ requires.
\begin{question}\label{Question: sacks}
    Consider any of the models obtained by forcing with iterated Sacks, side-by-side Sacks, iterated Silver, or product of Silver reals. Is every uniform ultrafilter over $\omega_1$ Tukey-top in any/all of those models?
\end{question} 
The Silver model is particularly interesting as it is conjectured (see~\cite[announcement 9 ff.]{no-p-pt-silver-ext}) that every ultrafilter over $\omega$ is Tukey-top there.
The following question concerns other possible constructions of non-Tukey-top ultrafilters:
\begin{question}
    Suppose that there is an $(\aleph_2,\aleph_2,\aleph_0)$-saturated ideal over $\omega_1$. Is there a non-Tukey-top uniform ultrafilter over $\omega_1$?
\end{question}
Another method for constructing weakly normal ultrafilters is due to Foreman, Magidor, and Shelah through layered ideals~\cite{MMpartII}~
\begin{question}
    Let $U$ be the FMS weakly normal ultrafilter from~\cite{MMpartII}. Is it non-Tukey-top?
\end{question}
\begin{question}\label{Question: weakly normal omega1}
    Is it consistent, relative to large cardinals, that there is a weakly normal Tukey-top ultrafilter over $\omega_1$?
\end{question}
As we have seen in subsection~\ref{Seciton: non-regular}, for $\kappa>\omega_1$ it is possible for weakly normal ultrafilters to be Tukey-top. 
One approach to a positive answer is to show that it is consistent relative to large cardinals 
that $(\dagger)_{I,\Ucal}$ holds for a weakly normal ultrafilter $\Ucal$.

The constructions we used for non-Tukey-top ultrafilters over $\omega_1$ require large cardinals. We do not know whether large cardinals are necessary:
\begin{question}
    Does the existence of a non-Tukey-top uniform ultrafilter over $\omega_1$ imply the consistency of large cardinals? What about non-$\omega_1$-Tukey-top uniform ultrafilters?
\end{question}
One approach to showing that it does, in alignment with our intuition and current examples indicating that non-Tukey-top ultrafilters are special and rare, is to connect such cardinals with non-regular ultrafilters:
\begin{question}
    Is every non-Tukey-top ultrafilter uniform ultrafilter over $\omega_1$ non-regular? What about non-$\omega_1$-Tukey-top ultrafilters?
\end{question}
The best-known lower bound for the consistency strength of the existence of non-regular ultrafilters over $\omega_1$, due to Deiser and Donder~\cite{DDonder}, is a stationary limit of measurable cardinals. The same question for $\kappa>\omega_1$ is of interest, where the best lower bound, due to Cox~\cite{Cox_2011}, is a measurable cardinal $\kappa$ with Mitchell order at least $\kappa^+$.
\begin{remark}
    Answering this question in the positive will in particular show that a counterexample to Kunen's problem requires large cardinals.
\end{remark}
Finally, it is natural to ask what influence strong forcing axioms have on Isbell's problem for $\omega_1$:
\begin{question}
    Does PFA imply that every uniform ultrafilter over $\omega_1$ is Tukey-top? What about stronger forcing axioms such as MM?
\end{question}
\bibliographystyle{amsplain}
\bibliography{ref}
\end{document}